\RequirePackage{fix-cm}
\documentclass[letterpaper, 10pt]{article}

  \usepackage{hyperref}
\usepackage{amssymb}
\usepackage{amsmath}
\usepackage{amsthm}
\usepackage[cmtip,all]{xy}
\usepackage[pagewise]{lineno}
\usepackage{etoolbox}
\usepackage{thmtools,thm-restate}
\usepackage[toc,page]{appendix}

 \usepackage{tikz}
\usetikzlibrary{matrix}

\newtheorem{lemma}{Lemma}
\newtheorem{proposition}{Proposition}
\newtheorem{corollary}{Corollary}

 \theoremstyle{remark}
 \newtheorem{example}{Example}
 \newtheorem{remark}{Remark}

\DeclareMathOperator{\Ker}{Ker}
\DeclareMathOperator{\im}{Im}
 \DeclareMathOperator{\Hom}{Hom}
  \DeclareMathOperator{\Homb}{\bf{Hom}}
  \DeclareMathOperator{\Inf}{Inf}
   \DeclareMathOperator{\Res}{Res}
     \DeclareMathOperator{\Norme}{Norme}
   \DeclareMathOperator{\Trace}{Trace}
   \DeclareMathOperator{\Aut}{Aut}
 \DeclareMathOperator{\Out}{Out}
  \DeclareMathOperator{\Span}{Span}
    \DeclareMathOperator{\tr}{tr}
        \DeclareMathOperator{\Lie}{Lie}
    
            \DeclareMathOperator{\codim}{codim}
    \DeclareMathOperator{\Fix}{Fix}

\begin{document}

\title{Bad irreducible subgroups and singular locus for character varieties in $PSL(p,\mathbb{C})$}

\author{Cl\'ement Gu\'erin\thanks{University of Luxembourg, Campus Kirchberg
Mathematics Research Unit, 6, rue Richard Coudenhove-Kalergi, L-1359, Luxembourg. \textit{e-mail : clement.guerin@uni.lu}}}

\date{}

\maketitle

\begin{abstract}
We give the centralizers of irreducible representations from  a finitely generated group $\Gamma$ to $PSL(p,\mathbb{C})$ where $p$ is a prime number. This leads to a description of the singular locus (the set of conjugacy classes of representations whose centralizer strictly contains the center of the ambient group) of the irreducible part of the character variety $\chi^i(\Gamma,PSL(p,\mathbb{C}))$. When $\Gamma$ is a free group of rank $l\geq 2$ or the fundamental group of a closed Riemann surface of genus $g\geq 2$, we give a complete description of this locus and prove that this locus is exactly the set of algebraic singularities of the  irreducible part of the character variety.
\end{abstract}

\textbf{Keywords : }Representation variety $\cdot$ Character variety $\cdot$ Irreducible representations $\cdot$ Centralizer of irreducible representations $\cdot$ Orbifolds $\cdot$ Fuchsian groups representations

\textbf{Mathematics Subject Classification (2000)} 20C15 $\cdot$ 15A21

\tableofcontents

\section{Introduction}\label{intro}

Let $G$ be a complex reductive group, e.g. $GL(n,\mathbb{C})$, $SL(n,\mathbb{C})$ or $PSL(n,\mathbb{C})$.  A proper subgroup $P$ of $G$ is \textit{parabolic} if $G/P$ is a complete variety. A subgroup $H$ of $G$ is \textit{irreducible} if $H$ is not contained in a parabolic subgroup of $G$.  A subgroup $H$ of $G$ is \textit{completely reducible}, if for any parabolic subgroup $P$ of $G$ containing $H$, there is a Levi subgroup $L$ of $P$ such that $H$ is contained in $L$. 

Throughout the paper, if $G$ is a group, $Z(G)$ denotes its center and if $H$ is a subgroup of $G$, $Z_G(H)$ denotes the centralizer of $H$ in $G$. If $H$ is an irreducible subgroup of $G$, we say that $H$ is \textit{good} if $Z_G(H)=Z(G)$ and \textit{bad} else. Sikora proved (see \cite{Sik}, Corollary 17) that a completely reducible subgroup $H$ of $G$ is irreducible if and only if its centralizer $Z_G(H)$ in $G$ is a finite extension of $Z(G)$.

\bigskip

Let $\Gamma$ be a finitely generated group,  a representation from $\Gamma$ to $G$ is \textit{irreducible} (resp. \textit{completely reducible}, \textit{good}, \textit{bad}) if $\rho(\Gamma)$ is. While Schur's Lemma implies that all irreducible representations in $SL(n,\mathbb{C})$ or $GL(n,\mathbb{C})$ are good, there are many examples in the literature of bad representations in other complex reductive algebraic groups (see Proposition 3.32 in \cite{F-L}). Roughly speaking, our goal in this paper is to study bad representations from a finitely generated group to $PSL(p,\mathbb{C})$ when $p$ is a prime number.  We make a brief recall of some notions on representation/character varieties (see  \cite{Sik} and also \cite{L-M} for a complete exposition). 

\bigskip

The set of representations from $\Gamma$ to $G$ is denoted $\Hom(\Gamma,G)$. It is an affine algebraic set. Indeed, if $\Gamma$ is generated by $r$ elements, $\Hom(\Gamma,G)$ can be cut out from $G^r$ by polynomial equations. As a result, there is a \textit{universal representation algebra} $R(\Gamma,G)$ (see Paragraph 5 in \cite{Sik}) such that $\Hom(\Gamma,G)$ corresponds to the set of maximal ideals of $R(\Gamma,G)$. Remark that $R(\Gamma,G)$ may not be reduced (i.e. it could non-trivial contain nilpotent elements). As a result, it is useful to define the \textit{schematic representation variety} $\Homb(\Gamma,G)$ as the spectrum of the ring $R(\Gamma,G)$. By definition, the set of $\mathbb{C}$-points of $\Homb(\Gamma,G)$ is $\Hom(\Gamma,G)$. From Section \ref{badsub} to Section \ref{surfgrpcase} included, we will only be interested in the $\mathbb{C}$-points of the representation variety. However, the tangent spaces computed in Section \ref{orbalgsing}, are, a priori, tangent  to the schematic representation variety $\Homb(\Gamma,G)$ and it is necessary to introduce this scheme for this reason. There are two natural topologies on $\Hom(\Gamma,G)$. One is the Zariski topology, the other is the transcendental topology (induced by the topology of $G$ as a complex  group). In this paper, we will mostly consider the second topology.

We denote $\Hom^i(\Gamma,G)$ the set of irreducible representations from $\Gamma$ to $G$. Since being non-irreducible is a closed condition, the set $\Hom^i(\Gamma,G)$ is   open in $\Hom(\Gamma,G)$ (see Proposition 27 in \cite{Sik}).

The adjoint group $G/Z(G)$ acts on $\Hom(\Gamma,G)$ by conjugation.  If $\rho:\Gamma\rightarrow G$ is a representation, its conjugacy class will be denoted $[\rho]$. Roughly speaking, the character variety is the quotient of $\Hom(\Gamma,G)$ by this conjugation action. However, in order to obtain an interesting geometric structure on the quotient we need to consider the GIT quotient $\Hom(\Gamma,G)//(G/Z(G))$. It is only defined  when $G$ is reductive. This quotient will be called the \textit{character variety} of $\Gamma$ into $G$ and will be denoted $\chi(\Gamma,G)$.  By definition, there is an algebraic map $\Psi: \Hom(\Gamma,G)\rightarrow \chi(\Gamma,G)$ which induces by duality an isomorphism between $\mathbb{C}[\chi(\Gamma,G)]$ and $\mathbb{C}[\Hom(\Gamma,G)]^G=(R(\Gamma,G)/\sqrt{(0)})^G$. Likewise, one may define the \textit{schematic character variety} $\mathfrak{X}(\Gamma,G)$ as the spectrum of  $\mathbb{C}[\Homb(\Gamma,G)]^G=R(\Gamma,G)^G$.

It turns out that $\chi(\Gamma,G)$ can be identified to the set of closed orbits in $\Hom(\Gamma,G)$ for the conjugation action. The projection $\Psi$ sends a representation $\rho$ to the unique closed orbit contained in the topological closure of $[\rho]$. Since $[\rho]$ is closed if and only if $\rho$ is completely reducible, $\Psi$ restricted to the subset of completely reducible representations is the usual projection on a quotient space. 

In particular, if we denote $\chi^i(\Gamma,G)$ the subset of conjugacy classes of irreducible representations in $\chi(\Gamma,G)$, then  $\chi^i(\Gamma,G)$ can be identified to the usual quotient $\Hom^i(\Gamma,G)/(G/Z(G))$. Remark that the action of $G/Z(G)$ on $\Hom^i(\Gamma,G)$ is proper by Proposition 1.1 in \cite{J-M}. In our paper, the \textit{singular locus of the character variety} of $\Gamma$ into $G$ is the set $\chi^i_{Sing}(\Gamma,G)$ of conjugacy classes of bad representations. 

A closed point $x$ in a (possibly non-reduced) algebraic scheme $X$ is \text{simple} if it belongs to a unique irreducible component and the dimension of its Zariski-tangent space  coincides with the dimension of the unique irreducible component containing it. A point $x$ in $X$ is an \textit{algebraic singularity} in $X$ if it is not a simple point. In particular, $\chi^i(\Gamma,G)$ or $\mathfrak{X}^i(\Gamma,G)$ may contain algebraic singularities.

Remark that being a simple point of the schematic  representation (resp. character) variety is  stronger than being a simple point of the representation (resp. character) variety, see Paragraph 9 in \cite{Sik}. Following its terminology, we say that a representation is \textit{scheme smooth} (resp. \textit{smooth}) if it is a simple point of $\Homb(\Gamma,G)$ (resp. $\Hom(\Gamma,G)$). We use the same terminology for conjugacy class of representations.

The question whether the singular locus of the character variety coincides with the set of algebraic singularities of the irreducible part of the  (schematic) character variety or not is discussed in Section \ref{orbalgsing}. Proposition \ref{singvarcar} states that for $G=PSL(p,\mathbb{C})$ and $\Gamma$ a free group of rank $\geq 2$ or a closed surface group of genus $\geq 2$, these sets are the  same. This statement cannot be generalized to any finitely generated group $\Gamma$ and complex reductive group $G$ because of  Examples  \ref{algnotorb} and \ref{orbnotalg}.

\bigskip

When $\Gamma$ is a free group over $l\geq2$ generators, then  $\Hom(\Gamma,G)=G^l$ and $\Hom^i(\Gamma,G)$ being open in it, is a manifold. When $\Gamma$ is a closed surface group of genus $g\geq 2$ and $G$ is reductive, it can be shown that $\Hom^i(\Gamma,G)$ is a manifold (all its points are scheme smooth, c.f. Proposition 37 in \cite{Sik}, see also Paragraph 1.2 in  \cite{Gol}).  Since the action of $G/Z(G)$ on $\Hom^i(\Gamma,G)$ is proper and its stabilizers are finite, the quotient space $\chi^i(\Gamma,G)$ is, in these cases, an orbifold. The singular locus we defined above is   the orbifold locus of  $\chi^i(\Gamma,G)$, i.e. the set of points with a non-trivial local isotropy. This geometric interpretation as the set of orbifold singularities of an orbifold is our first reason for studying this singular locus. However, it appeared that this locus was well defined for any finitely generated group and that most results  we obtained were true for any finitely generated group, even if $\chi^i(\Gamma,G)$ is not an orbifold. 

\bigskip

From Section \ref{badsub} to the end of the paper, we only focus on the case $G=PSL(p,\mathbb{C})$ with $p$ a prime number.  For some proofs, it is necessary to deal with the case $p=2$ separately from the case $p$ odd. When this is the case, we will systematically assume  $p$  odd, leaving $p=2$ to the reader. Anyway, for $p=2$, most of these results have already been obtained by Heusener and Porti in \cite{H-P}.

In the sequel we will need a few notations. Fix a prime number $p$. Denote $\pi$ the natural projection from $SL(p,\mathbb{C})$ to $PSL(p,\mathbb{C})$. To simplify some proofs, it is easier to see matrices $M\in SL(p,\mathbb{C})$ as automorphisms of $\mathbb{C}^{\mathbb{Z}/p}$ with the canonical basis indexed by $\mathbb{Z}/p$. In particular their lines and rows will be indexed by $0,\dots, p-1$.

For any matrix $M\in SL(p,\mathbb{C})$, we will denote $\overline{M}$ its projection in $PSL(p,\mathbb{C})$.   $D$ will denote the subgroup of diagonal matrices in $SL(p,\mathbb{C})$ and $\xi$ will be a primitive $p^{\text{th}}$ root of the unity. When $p$ is odd, the diagonal matrix $D(\xi)$ is defined by $(D(\xi))_{i,i}:=\xi^{i}$ for $0\leq i\leq p-1$.  To any bijection $\sigma$ of $\mathbb{Z}/p$, we associate the corresponding permutation matrix $M_{\sigma}\in GL(p,\mathbb{C})$. Finally,  $c$ will denote the cyclic permutation $(0,1,\dots,p-1)$ of $\mathbb{Z}/p$. When $p=2$,  we define $D(\xi):=\begin{pmatrix}\sqrt{-1}&0\\0&-\sqrt{-1}\end{pmatrix}$ and $M_c:= \begin{pmatrix}0&\sqrt{-1}\\ \sqrt{-1}&0\end{pmatrix}$.  We prove, in Section \ref{badsub} :

\begin{restatable}{theorem}{classifthm}\label{classifcentr}
Let $p$ be a prime number and $\overline{H}$ be a bad subgroup of $PSL(p,\mathbb{C})$. Then, there are two cases :

\begin{itemize}

\item  $Z_{PSL(p,\mathbb{C})}(H)$ is isomorphic to $\mathbb{Z}/p\times\mathbb{Z}/p$. In which case $\overline{H}$ is conjugate to the group $\langle \overline{D(\xi)}\rangle\times \langle\overline{M_c}\rangle$ and $Z_{PSL(p,\mathbb{C})}(\overline{H})=\overline{H}$.

\item  $Z_{PSL(p,\mathbb{C})}(H)$ is isomorphic to $\mathbb{Z}/p$. In which case $\overline{H}$ is conjugate to  $\overline{K}\rtimes \langle\overline{M_c}\rangle$ where $\overline{K}$ is a non-trivial subgroup  of $\overline{D}$, different from $\langle \overline{D(\xi)}\rangle$ and invariant by the action of $\langle\overline{M_c}\rangle$, in particular $Z_{PSL(p,\mathbb{C})}(\overline{H})$ is conjugate to $\langle \overline{D(\xi)}\rangle$. 
\end{itemize}

\end{restatable}

For $p=2$, this is implied by   Remark  3.11  and Remark 4.3 in \cite{H-P}. As a result, any bad representation from $\Gamma$ to $PSL(p,\mathbb{C})$ is conjugate to a representation into $\overline{D}\rtimes \langle\overline{M_c}\rangle$. Therefore, the natural inclusion of $\Hom^i(\Gamma,\overline{D}\rtimes \langle \overline{M_c}\rangle)$ into $\Hom^i(\Gamma,PSL(p,\mathbb{C}))$ induces a surjective map $\varphi$ from $\chi^i(\Gamma,\overline{D}\rtimes \langle \overline{M_c}\rangle)$ onto $\chi^i_{Sing}(\Gamma,PSL(p,\mathbb{C}))$. This will give a   parametrization of the singular locus.

For any $\overline{\rho}$ in $\Hom(\Gamma,\langle \overline{M_c}\rangle)$, define $\mathcal{H}_{\overline{\rho}}^i:=\{\rho\in \Hom^i(\Gamma,\overline{D}\rtimes \langle \overline{M_c}\rangle)\mid q\circ \rho=\overline{\rho}\}$ and $\overline{\mathcal{H}_{\overline{\rho}}}^i$ be its projection on the character variety. Then  $\chi^i(\Gamma,\overline{D}\rtimes \langle \overline{M_c}\rangle)$ is the union of all $\overline{\mathcal{H}_{\overline{\rho}}}^i$. In section \ref{diagbyfin}, we give a cohomological description for $\overline{\mathcal{H}_{\overline{\rho}}}^i$. In  Section \ref{secinj}, we prove :

\begin{restatable}{theorem}{homeobarHrho}\label{homeobarHrho}

Let $\Gamma$ be a finitely generated group,  $p$ be a prime number and  $\overline{\rho}$    be a non-trivial morphism in $\Hom(\Gamma,\langle \overline{M_c}\rangle)$. Then $\varphi_{|\overline{\mathcal{H}_{\overline{\rho}}}^i}$ is a homeomorphism onto its image.

\end{restatable}

As a result, $\chi^i_{Sing}(\Gamma,PSL(p,\mathbb{C}))$ is a union of a finite number of $\varphi(\overline{\mathcal{H}_{\overline{\rho}}}^i)$ whose geometry is given by cohomology groups. At the end of Section \ref{secinj}, we briefly justify that we only need to chose, for any normal subgroup $K$ of index $p$ in $\Gamma$, one $\overline{\rho}_K$ in $\Hom(\Gamma,\langle \overline{M_c}\rangle)$ whose kernel is $K$ to get $\chi^i_{Sing}(\Gamma,PSL(p,\mathbb{C}))$. To sum up :
$$\chi^i_{Sing}(\Gamma,PSL(p,\mathbb{C}))=\bigcup_{\substack{K\triangleleft \Gamma\\   \text{$[\Gamma:K]=p$} } }\varphi(\overline{\mathcal{H}_{\overline{\rho}_K}}^i)$$

The subset $\varphi(\overline{\mathcal{H}_{\overline{\rho}_K}}^i)$  will be called the \textit{pseudo-component} associated to $K$. When $\Gamma$ is a free group these are the irreducible components of the singular locus (this is not true for all $\Gamma$). In Section \ref{secint}, we study the possible intersections between different pseudo-components. Recall that the $p$-rank of $\Gamma$ is by definition $\dim_{\mathbb{Z}/p}(\Gamma^{Ab}/p\Gamma^{Ab})$.

\begin{restatable}{theorem}{combi}\label{combi}

Let $\Gamma$ be a finitely generated group,  $p$ be a prime number. Let $r$ be the $p$-rank of $\Gamma$. Then : 
\begin{enumerate}
\item $\chi^i_{Sing}(\Gamma,PSL(p,\mathbb{C}))$ is the union of $\frac{p^r-1}{p-1}$ pseudo-components.

\item The intersection of two different pseudo-components is finite of cardinal $p-1$. All its elements are conjugacy classes of abelian irreducible representations.

\item Conversely, any conjugacy class of abelian irreducible representations belongs to exactly $p+1$ pseudo-components.

\item There are exactly $\frac{(p^r-1)(p^{r-1}-1)}{p^2-1}$ conjugacy classes  of abelian irreducible representations from $\Gamma$ to $PSL(p,\mathbb{C})$. 
 \end{enumerate}

\end{restatable}

One can compare this theorem  to the well-known example when $p=2$ and $\Gamma=\mathbb{F}_2=\langle a,b\rangle$ is a free group over $2$ generators (see Example 4.4 in \cite{H-P} or Corollary 10 in \cite{Sik15}).

\begin{example}\label{F2PSL2} First, the natural map $\pi^*$ from $SL(2,\mathbb{C})^2=\Hom(\mathbb{F}_2,SL(2,\mathbb{C}))$ to $PSL(2,\mathbb{C})^2=\Hom(\mathbb{F}_2,PSL(2,\mathbb{C}))$ is onto. The automorphism group of this orbifold cover is  the Klein group $(\mathbb{Z}/2)^2$ : $(i,j)\cdot (A,B):=((-1)^iA,(-1)^jB)$.

On the other hand, an old result of Vogt states that $\chi(\mathbb{F}_2,SL(2,\mathbb{C}))=\mathbb{C}^3$ by sending $[\rho]$ to $(\tr(\rho(a)),\tr(\rho(b)),\tr(\rho(ab)))$.

Since $\pi^*$ induces a map from $\chi(\mathbb{F}_2,SL(2,\mathbb{C}))=\mathbb{C}^3$ to $\chi(\mathbb{F}_2,PSL(2,\mathbb{C}))$, we can identify $\chi(\mathbb{F}_2,PSL(2,\mathbb{C}))$ to $\chi(\mathbb{F}_2,SL(2,\mathbb{C}))/K$ where the action of $K$ is given by $(i,j)\cdot(x,y,z):=((-1)^ix,(-1)^jy,(-1)^{i+j}z) $. 

Computing the invariant ring  $\mathbb{C}[x,y,z]^{(\mathbb{Z}/2)^2}$, one identifies   $\chi(\mathbb{F}_2,PSL(2,\mathbb{C}))$ to $\{(X,Y,Z,T)\in\mathbb{C}^4\mid T^2=XYZ\}$ and explicitly computes $\pi^*(x,y,z)$ to be $(x^2,y^2,z^2,xyz)$. Removing three non-irreducible points, the set of   branch points for $\pi^*$  is, by definition,  the singular locus  of the orbifold $\chi^i_{Sing}(\mathbb{F}_2,PSL(2,\mathbb{C}))$.

The set of branch points is easily computed to be the union of $X=Z=T=0$, $Y=Z=T=0$ and $X=Y=T=0$. Each of them (minus one non-irreducible point) corresponds to one of  the three  pseudo-components in the theorem above. Any two of them intersect in one single point whose coordinates are   $X=Y=Z=T=0$. It corresponds to the abelian irreducible representation $a\mapsto \overline{D(\xi)}$, $b\mapsto \overline{M_c}$ and thus  fully agrees with the aforementioned theorem.

 \end{example} 

When  $\Gamma=\mathbb{F}_l$ is a free group of rank $l\geq 2$, we prove in Section \ref{freegrpcase} (see Corollary \ref{orblocfree}) that $\chi^i_{Sing}(\mathbb{F}_l,PSL(p,\mathbb{C}))$ is connected of dimension $(p-1)(l-1)$. 

In \cite{F-L-R}, Florentino, Lawton and Ramras  compute higher homotopy groups of the irreducible part of  free groups character  varieties into $G_n:=SL(n,\mathbb{C})$ or $GL(n,\mathbb{C})$ : $\pi_k(\chi^i(\mathbb{F}_l,G_n))$  (see Theorem 5.4 in loc. cit.).

In Remark 5.7, they conjecture that for any complex reductive group $G$ and $l$ big enough, $\pi_2(\chi^{good}(\mathbb{F}_l,G))=\pi_1(G/Z(G))$ provided that we can bound above the dimension of bad representations.  The dimension count we obtained directly leads to the validity of the conjecture when $G=PSL(p,\mathbb{C})$ for $(l-1)(p-1)\geq 2$.

\bigskip

In Section \ref{surfgrpcase}, we study the case when $\Gamma=\pi_1(\Sigma_g)$ is a closed surface group  of genus $g\geq 2$.  Corollary \ref{conncomporblocsurfgrp} states that $\chi^i_{Sing}(\pi_1(\Sigma_g),PSL(p,\mathbb{C}))$ has $p$ connected components (given by the Euler invariant).  Jun Li's Theorem in \cite{Li} implies that  $\pi_0(\chi(\pi_1(\Sigma_g),G))$ is equal to $\pi_1(G)$. Therefore, when $G=PSL(p,\mathbb{C})$, each connected component of $\chi(\pi_1(\Sigma_g),PSL(p,\mathbb{C}))$ contains one unique connected component of  $\chi^i_{Sing}(\pi_1(\Sigma_g),PSL(p,\mathbb{C}))$.

\section{Bad subgroups   in $PSL(p,\mathbb{C})$}\label{badsub}

In this section, we classify, up to conjugation, bad subgroups in $PSL(p,\mathbb{C})$. More precisely, we prove :

 \classifthm*

The main idea to prove this theorem is to pull-back the problem in $SL(p,\mathbb{C})$ using $\pi$. For $A\in SL(p,\mathbb{C})$ and $\lambda\in \mathbb{C}$, we denote $E_{\lambda}(A)=\{v\in \mathbb{C}^{\mathbb{Z}/p}\mid Av=\lambda v\}$. 

\begin{lemma}\label{eigspace}
Let $n\geq 1$  and $A,B$ be two matrices in $GL(n,\mathbb{C})$. Assume $[A,B]=\lambda I_n$ where $\lambda \in \mathbb{C}^*$.Then, for $\mu \in \mathbb{C}$, we have $BE_{\mu}(A)=E_{\lambda\mu}(A)$. 
\end{lemma}

\begin{proof}
The assumption implies that $BA=\lambda^{-1}AB$. For any $v\in \mathbb{C}^{\mathbb{Z}/n}$, we have the following equivalences : $v\in E_{\mu}(A)\Leftrightarrow Av=\mu v\Leftrightarrow BAv=\mu Bv\Leftrightarrow ABv=\lambda \mu Bv$ i.e. $Bv\in E_{\lambda\mu}(A)$. Therefore, $BE_{\mu}(A)=E_{\lambda\mu}(A)$.  \end{proof}

This directly leads to :

\begin{proposition}\label{central}
Let $p$ be a prime number and $\overline{H}$ be a bad subgroup of $PSL(p,\mathbb{C})$. Then, up to conjugation, $\overline{D(\xi)}$ belongs to $Z_{PSL(p,\mathbb{C})}(\overline{H})$. 
\end{proposition}

\begin{proof}
Assume $p$   odd. Let $H$ be $\pi^{-1}(\overline{H})$ and $U$ be $\pi^{-1}(Z_{PSL(p,\mathbb{C})}(\overline{H}))$. Since $Z_{PSL(p,\mathbb{C})}(\overline{H})$ is non-trivial, it follows that $U$ contains an element $u$ which is not central. Furthermore, since $H$ is an irreducible subgroup of $SL(p,\mathbb{C})$, its centralizer is $Z(SL(p,\mathbb{C}))$ by Schur's lemma. 

As a result, there exists $h_0\in H$ such that $[h_0,u]\neq I_p$. However, since $\overline{h_0}$ and $\overline{u}$ commute, it follows that $[h_0,u]$ belongs to $Z(SL(p,\mathbb{C}))$. Therefore, there exists $0<k<p$ such that $[h_0,u]=\xi^kI_p$.

Applying Lemma \ref{eigspace}, $h_0$ acts on the spectrum of $u$ by multiplying by $\xi^k$. Let $\mu$ be an eigenvalue of $u$, since $\xi^k$ is a non-trivial primitive $p^{\text{th}}$ root of the unity, $u$ has $p$ different eigenvalues $\mu,\xi^k\mu,\dots,\xi^{k(p-1)}\mu$. Finally, since $u$ is a matrix of dimension $p$, its eigenspaces have dimension $1$ and $u$ is conjugate to the diagonal matrix with $\mu,\xi^k\mu,\dots,\xi^{k(p-1)}\mu$ on the diagonal. Since $\det(u)=1$ and $p$ is odd, we see that $\mu$ is a $p^{\text{th}}$ root of the unity. Therefore, $u$ is  conjugate to $D(\xi)$.  \end{proof}

Because of this proposition, we   compute the centralizer of $\overline{D(\xi)}$ in $PSL(p,\mathbb{C})$.

\begin{lemma}\label{centrdxi}
Let $p$ be a prime number. Then $Z_{PSL(p,\mathbb{C})}(\overline{D(\xi)})=\overline{D}\rtimes \langle \overline{M_c}\rangle$ where  $\langle \overline{M_c}\rangle$ acts on $\overline{D}$ by conjugation.
\end{lemma}
\begin{proof}
Assume $p$  odd. Let $U:=\pi^{-1}(Z_{PSL(p,\mathbb{C})}(\overline{D(\xi)}))$.  For $0\leq k\leq p-1$, we have  $[D(\xi),M_c^k]=\xi^k I_p$ by straightforward computations. It follows that $U$ contains both the group $D$ and the matrix $M_c$. Conversely, if $u\in U$ then $[D(\xi),u]=\xi^k I_p$ for some $k$ and therefore :
\begin{align*}
[D(\xi),uM_c^{-k}]&=D(\xi)uM_c^{-k}D(\xi)^{-1}M_c^ku^{-1}\\
&=D(\xi)uD(\xi)^{-1}\xi^{-k}u^{-1}\text{ since  $[D(\xi),M_c^{-k}]=\xi^{-k} I_p$}\\
&=\xi^{-k}[D(\xi),u]=I_p\text{.}\end{align*}

Hence, $u$ is the product of an element in the centralizer of $D(\xi)$ in $SL(p,\mathbb{C})$ (which is $D$, since $D(\xi)$ is diagonal with pairwise distinct eigenvalues) with some power of $M_c$. As a result,   $U=D\rtimes \langle M_c\rangle$ and projecting this equality in $PSL(p,\mathbb{C})$, we have $Z_{PSL(p,\mathbb{C})}(\overline{D(\xi)})=\overline{D}\rtimes \langle \overline{M_c}\rangle$. \end{proof}

When $p=2$, the proof is slightly different since there is a non-trivial intersection between $D$ and $\langle M_c\rangle$, however $\overline{D}$ and $\langle\overline{M_c}\rangle$ still  have a trivial intersection.

\begin{proposition}\label{centralconj}
Let $p$ be a prime number and $\overline{H}$ be a bad subgroup of $PSL(p,\mathbb{C})$. Then, there is a non-trivial subgroup  $\overline{K}$ of $\overline{D}$ such that $\overline{H}$ is conjugate to $\overline{K}\rtimes\langle \overline{M_c}\rangle$. 
\end{proposition}

\begin{proof}
We denote $q$ the natural projection   $\overline{D}\rtimes \langle \overline{M_c}\rangle\to \langle \overline{M_c}\rangle$. Combining Proposition \ref{central} and Lemma \ref{centrdxi}, $\overline{H}$ is conjugate to a subgroup of $\overline{D}\rtimes \langle \overline{M_c}\rangle$. We identify $\overline{H}$ to this subgroup of $\overline{D}\rtimes \langle \overline{M_c}\rangle$. 

If $q(\overline{H})$ were trivial then $\overline{H}$ would be contained in $\overline{D}$ which is not irreducible, therefore $q(\overline{H})$ is not trivial. Let $x$ be in $\pi^{-1}(\overline{H})$ such that $q(\overline{x})=\overline{M_c}$. Then :

$$x=\begin{pmatrix}\lambda_0&&\\&\ddots&\\&&\lambda_{p-1}\end{pmatrix}M_c\text{.} $$

If $s$ denotes a diagonal matrix where $s_{i,i}=\lambda_0^{-1}\dots\lambda_{i}^{-1}$ for $0\leq i\leq p-1$ then  $sxs^{-1}=M_c$. If $t:=(\det(s))^{1/p}s$ then $txt^{-1}=M_c$ and $t$ is diagonal of determinant $1$. Therefore, $\overline{H}$ is conjugate in $\overline{D}\rtimes \langle \overline{M_c}\rangle$ to  a subgroup $\overline{H'}$ of $\overline{D}\rtimes \langle \overline{M_c}\rangle$  which contains $\overline{M_c}$. Since $\overline{M_c}$ belongs to $\overline{H'}$, we have that $\overline{H'}=\Ker(q_{|\overline{H'}})\rtimes \langle \overline{M_c}\rangle$. Since $\overline{H'}$ needs to be irreducible, it is clear that $\Ker(q_{|\overline{H'}})$ cannot be trivial.   \end{proof}

 Similar to Lemma \ref{centrdxi} (although it is written in a different way since it will be used with Lemma \ref{freeMc} later),  we have :

\begin{lemma}\label{centrMc}
Let $p$ be a prime number and $g\in SL(p,\mathbb{C})$ such that  $[g,M_c]\in Z(SL(p,\mathbb{C}))$. Then, $g=P(M_c)D(\xi)^k$ where $P\in\mathbb{C}[X]$ and $k\geq 0$.
\end{lemma}

\begin{proof}Assume $p$  odd. Since $M_c$ has pairwise distinct eigenvalues (like $D(\xi)$, its minimal polynomial is   $X^p-1$), any matrix  commuting with $M_c$ can be written as a polynomial in $M_c$. Like in Lemma \ref{centrdxi}, if $[g,M_c]=\xi^k I_p$ then $[gD(\xi)^{-k},M_c]=I_p$ and therefore, what is written above implies that there is a polynomial $P$ such that $gD(\xi)^{-k}=P(M_c)$ and we are done. 
 \end{proof}

\begin{lemma}\label{freeMc}

Let $p$ be  a prime number. Let $(d_0,\dots,d_{p-1})$ and $(d_0',\dots,d_{p-1}')$ be two $p$-tuples of complex diagonal matrices. Then :

$$\sum_{j=0}^{p-1}d_jM_c^j= \sum_{j=0}^{p-1}d_j'M_c^j\Rightarrow \forall \text{ }0\leq j\leq p-1\text{, } d_j=d_j'\text{.}$$

\end{lemma}

\begin{proof}
Assume $p$  odd. For $0\leq j\leq p-1$, and for $0\leq i,k\leq p-1$, we have :

$$(d_jM_c^j)_{i,k}=\left\lbrace\begin{array}{ll} (d_j)_{i,i}&\text{ if $i-k=j$ mod $p$}\\0&\text{ else.}\end{array}\right.$$

As a result :

$$\left(\sum_{j=0}^{p-1}d_jM_c^j\right)_{i,k}=(d_{k-i\text{ mod } p})_{i,i}\text{.}$$

Applying this expression to both $(d_0,\dots,d_{p-1})$ and $(d_0',\dots,d_{p-1}')$, we easily see that the equality of the sum of matrices in the assumption implies the equalities $d_j=d_j'$ for $j=0,\dots,p-1$.  \end{proof}

\begin{proposition}\label{centrKMc}

Let $p$ be a prime number, $\overline{K}$ be a non-trivial subgroup of $\overline{D}$ which is stable by the conjugation action of $\overline{M_c}$. Then :

$$Z_{PSL(p,\mathbb{C})}(\overline{K}\rtimes\langle\overline{M_c}\rangle)=\left\lbrace\begin{array}{ll} \langle \overline{D(\xi)}\rangle \times\langle\overline{M_c}\rangle &\text{ if } \overline{K}=\langle \overline{D(\xi)}\rangle \\
\langle \overline{D(\xi)}\rangle &\text{ else.}\\
\end{array}\right.$$

\end{proposition}

\begin{proof}  Since $\overline{D(\xi)}$ commutes with $\overline{D}\rtimes\langle\overline{M_c}\rangle$, $\langle \overline{D(\xi)}\rangle$ is, in any case, included in $Z_{PSL(p,\mathbb{C})}(\overline{K}\rtimes\langle\overline{M_c}\rangle)$.

Let  $\overline{x}$ be in $Z_{PSL(p,\mathbb{C})}(\overline{K}\rtimes\langle\overline{M_c}\rangle)$. By  Lemma \ref{centrMc}, there exist complex numbers $a_0$,\dots, $a_{p-1}$ and $0\leq k\leq p-1$ such that  $x=(a_0+a_1M_c+\cdots+a_{p-1}M_c^{p-1})D(\xi)^k$. Let   $\overline{d}$ be in $\overline{K}$. Then, there is $0\leq t\leq p-1$ such that  $dxd^{-1}=\xi^tx $. We also  have :

$$dxd^{-1}=\left(\sum_{j=0}^{p-1}a_jdM_c^jd^{-1} \right)D(\xi)^k=\left(\sum_{j=0}^{p-1}\underbrace{a_jdM_c^jd^{-1}M_c^{-j}}_{\text{ diagonal}}M_c^j \right)D(\xi)^k\text{.}$$

Therefore, we apply Lemma \ref{freeMc} to get $a_jdM_c^jd^{-1}M_c^{-j}=a_j\xi^t$ for $j=0,\dots,p-1$. Assume  $a_j$ is non-zero for some $j>0$, then  $[d,M_c^j]$ belongs to $Z(SL(p,\mathbb{C}))$ and therefore, $\overline{d}$ commutes with $\overline{M_c}^j$. Since $j\neq 0$ and $\overline{M_c}$ is of order $p$, $\overline{d}$ commutes with $\overline{M_c}$. Since $d$ is diagonal, Lemma \ref{centrMc} implies that $\overline{d}$ belongs to $\langle \overline{D(\xi)}\rangle$.

\bigskip

As a result, if $\overline{K}$ is not equal to $\langle \overline{D(\xi)}\rangle$ then $a_j=0$ for all $j>0$ and $\overline{x}$ belongs to $\langle \overline{D(\xi)}\rangle$. It follows that the centralizer of $\overline{K}\rtimes\langle\overline{M_c}\rangle$ in $PSL(p,\mathbb{C})$ is $\langle \overline{D(\xi)}\rangle$.

If $\overline{K}=\langle \overline{D(\xi)}\rangle$ then taking $d=D(\xi)$ above, we have $a_j[D(\xi),M_c^j]=a_j\xi^t$. Therefore $a_j(\xi^j-\xi^t)=0$. As a result, $a_j=0$ if $j\neq t$ and therefore $\overline{x}$ belongs to $ \langle \overline{D(\xi)}\rangle \times\langle\overline{M_c}\rangle $. We proved that the centralizer of  $ \langle \overline{D(\xi)}\rangle \times\langle\overline{M_c}\rangle $ is contained in $ \langle \overline{D(\xi)}\rangle \times\langle\overline{M_c}\rangle $, since this group is abelian, its centralizer in $PSL(p,\mathbb{C})$ is itself. \end{proof}

\begin{remark}\label{irred}   Let $\overline{K}$ be a non-trivial $\langle\overline{M_c}\rangle$-invariant subgroup of $\overline{D}$. Then $\overline{K} \rtimes\langle\overline{M_c}\rangle$ is completely reducible.  Since  they also have finite centralizers (using this proposition), Corollary 17 in \cite{Sik} implies that they are irreducible. \end{remark}

\begin{proof}[Theorem \ref{classifcentr}] It is a direct consequence of  Proposition \ref{centralconj} which states that  any bad  subgroup $\overline{H}$ is conjugate to  $\overline{K}\rtimes \langle\overline{M_c}\rangle$ where $\overline{K}$ is a non-trivial  subgroup of $\overline{D}$  and of Proposition \ref{centrKMc} which gives  their centralizer.   \end{proof}

Our next order of business is to describe   $\chi^i_{Sing}(\Gamma, PSL(p,\mathbb{C}))$ when $\Gamma$ is a finitely generated group. Denote $\Hom^i(\Gamma,\overline{D}\rtimes \langle \overline{M_c}\rangle)$ the set of irreducible representations from $\Gamma$ to $PSL(p,\mathbb{C})$ whose image is contained in $\overline{D}\rtimes \langle \overline{M_c}\rangle$. By definition, this set is  included in $\Hom^i(\Gamma,PSL(p,\mathbb{C})) $. We denote $\iota$ the inclusion. It induces   a map $\varphi$ on the character variety and we have  the following commutative diagram :
$$\xymatrix{\Hom^i(\Gamma,\overline{D}\rtimes \langle \overline{M_c}\rangle) \ar[rr]^{\iota}\ar[d]^{\textnormal{mod }\overline{D}\rtimes \langle \overline{M_c}\rangle}&&\Hom^i(\Gamma,PSL(p,\mathbb{C})) \ar[d]^{\textnormal{mod }PSL(p,\mathbb{C})}\\ \chi^i(\Gamma,\overline{D}\rtimes \langle \overline{M_c}\rangle)\ar[rr]^{\varphi} &&\chi^i(\Gamma,PSL(p,\mathbb{C}))}$$

From Theorem \ref{classifcentr}, we immediately deduce the following corollary :

\begin{corollary}\label{singular}

Let $p$ be a prime number and $\Gamma$ be a finitely generated group. Then  $\chi^i_{Sing}(\Gamma,PSL(p,\mathbb{C}))=\varphi(\chi^i(\Gamma,\overline{D}\rtimes \langle \overline{M_c}\rangle))$.

\end{corollary}

  Therefore, to describe the singular locus, it suffices to describe $\chi^i(\Gamma,\overline{D}\rtimes \langle \overline{M_c}\rangle)$ (Section \ref{diagbyfin}) and the behavior of $\varphi$ (Sections \ref{secinj} and \ref{secint}).

\section{The character variety into a virtually abelian semidirect product}\label{diagbyfin}

This section will be devoted to the description of $\chi^i(\Gamma,\overline{D}\rtimes \langle \overline{M_c}\rangle)$. This requires  group cohomology and its computation in terms of cochains. Definitions and results needed for this paper are given in Appendix \ref{grocoh}. 

We recall that $\langle \overline{M_c}\rangle$ naturally acts on $\overline{D}$ by conjugation. Let $q$ be the natural projection of the semidirect product $\overline{D}\rtimes \langle \overline{M_c}\rangle$ onto $\langle \overline{M_c}\rangle$. 

Given a group morphism $\overline{\rho}$ in $\Hom(\Gamma,\langle \overline{M_c}\rangle)$, we remark that it  makes of $\overline{D}$ a multiplicative $\Gamma$-module since $\langle \overline{M_c}\rangle$ acts by conjugation on $\overline{D}$. When it is necessary to specify the action, we will denote $\overline{D}_{\overline{\rho}}$ this $\Gamma$-module.

\bigskip

For any $\overline{\rho}$ in $\Hom(\Gamma,\langle \overline{M_c}\rangle)$, define $\mathcal{H}_{\overline{\rho}}:=\{\rho\in \Hom(\Gamma,\overline{D}\rtimes \langle \overline{M_c}\rangle)\mid q\circ \rho=\overline{\rho}\}$. The representation variety is a disjoint union of such sets :

$$\Hom(\Gamma,\overline{D}\rtimes \langle \overline{M_c}\rangle)=\bigcup_{\overline{\rho}\in \Hom(\Gamma,\langle \overline{M_c}\rangle)}\mathcal{H}_{\overline{\rho}}\text{.}$$

Furthermore :

\begin{proposition}\label{Hrho}

Let $\Gamma$ be a finitely generated group, $p$ be a prime number and $\overline{\rho}$ be in $\Hom(\Gamma,\langle \overline{M_c}\rangle)$. Then the following map is a well-defined homeomorphism 

$$f:\left|\begin{array}{clc}Z^1(\Gamma,\overline{D}_{\overline{\rho}})&\longrightarrow& \mathcal{H}_{\overline{\rho}}\\u&\longmapsto& \left(\gamma\mapsto u(\gamma)\overline{\rho}(\gamma)\right)\end{array}\right.\text{.}$$

\end{proposition}

\begin{proof}

Any group morphism  $\rho$ from $\Gamma$ to $\overline{D}\rtimes \langle \overline{M_c}\rangle  $ can be uniquely written as the product of a map $u_{\rho}:\Gamma\rightarrow \overline{D}$ and $q\circ \rho$. In particular, for any $\rho\in  \mathcal{H}_{\overline{\rho}}$ : $\rho(\cdot)=u_{\rho}(\cdot)\overline{\rho}(\cdot)$. 

It boils down to understand under which conditions on $u:\Gamma \rightarrow \overline{D}$ the map $\rho$ from $\Gamma$ to $\overline{D}\rtimes \langle \overline{M_c}\rangle$ defined as the product of $u$ and $\overline{\rho}$ is a group morphism. For two elements $\gamma,\gamma'$ in  $\Gamma$, we compute :
\begin{align*}
&\rho(\gamma\gamma')=u(\gamma\gamma')\overline{\rho}(\gamma\gamma')\\
\text{ and }&\rho(\gamma)\rho(\gamma')=u(\gamma)\overline{\rho}(\gamma)u(\gamma')\overline{\rho}(\gamma')=u(\gamma)\gamma\cdot  u(\gamma')\overline{\rho}(\gamma)\overline{\rho}(\gamma')\text{.}
\end{align*}

As a result, $\rho$ is a group morphism if and only if $u(\gamma\gamma')=u(\gamma) \gamma\cdot u(\gamma') $ for all $\gamma,\gamma'$ in  $\Gamma$, i.e. if and only if $u$ is a $1$-cocycle from  $\Gamma$ to $\overline{D}_{\overline{\rho}}$. It follows that :

$$\begin{array}{clc}Z^1(\Gamma,\overline{D}_{\overline{\rho}})&\longrightarrow& \mathcal{H}_{\overline{\rho}}\\u&\longmapsto& \left(\gamma\mapsto u(\gamma)\overline{\rho}(\gamma)\right)\end{array}$$

is a well-defined bijection. This map is   continuous and its inverse is the push-forward of the projection on the first factor for $\overline{D}\rtimes \langle \overline{M_c}\rangle$ which is, also continuous.  \end{proof}

There is a similar decomposition for the character variety. For $g\in \overline{D}\rtimes \langle \overline{M_c}\rangle$ and $\rho$ a representation from $\Gamma$ to $ \overline{D}\rtimes \langle \overline{M_c}\rangle$, $q\left(g\rho(\cdot) g^{-1}\right)=q(\rho(\cdot))$. As a result,   $\overline{\mathcal{H}_{\overline{\rho}}}:=\{[\rho]\in \chi(\Gamma,\overline{D}\rtimes \langle \overline{M_c}\rangle)\mid q\circ \rho=\overline{\rho}\}$ is well-defined and   :
 
 $$\chi(\Gamma,\overline{D}\rtimes \langle \overline{M_c}\rangle)=\bigcup_{\overline{\rho}\in \Hom(\Gamma,\langle \overline{M_c}\rangle)}\overline{\mathcal{H}_{\overline{\rho}}}\text{.}$$
 
Since the actions of  $\Gamma$ and $\langle \overline{M_c}\rangle$ on $\overline{D}$  commute, the conjugation action of $\langle \overline{M_c}\rangle$ on  $\overline{D}$ induces an action on $H^1(\Gamma,\overline{D}_{\overline{\rho}})$ (by acting on the coefficients). 
\begin{proposition}\label{barHrho}Let $\Gamma$ be a finitely generated group, $p$ be a prime number and $\overline{\rho}$ be in $\Hom(\Gamma,\langle \overline{M_c}\rangle)$. Then the following map is a well-defined homeomorphism  

$$\overline{f}:\left|\begin{array}{clc}H^1(\Gamma,\overline{D}_{\overline{\rho}})/\langle \overline{M_c}\rangle&\longrightarrow& \overline{\mathcal{H}_{\overline{\rho}}}\\ \left[u\right]\text{ mod }\langle \overline{M_c}\rangle&\longmapsto& \left[\gamma\mapsto u(\gamma)\overline{\rho}(\gamma)\right]\end{array}\right.\text{.}$$
\end{proposition}
\begin{proof}

Let $u_1$ and $u_2$ be in $Z^1(\Gamma,\overline{D}_{\overline{\rho}})$ and $\rho_i:=u_i\overline{\rho}$ be the corresponding element in $\mathcal{H}_{\overline{\rho}}$. Then $\rho_1$ is conjugate to $\rho_2$ in $\overline{D}\rtimes \langle \overline{M_c}\rangle$ if and only if there is $g=\overline{d}\text{ }\overline{M_c}^k$ in $\overline{D}\rtimes \langle \overline{M_c}\rangle$ such that $g\rho_1(\gamma)g^{-1}=\rho_2(\gamma)$ for all $\gamma\in \Gamma$. This is equivalent to $u_2(\gamma)=\overline{d}\gamma\cdot \overline{d}^{-1}\overline{M_c}^ku_1(\gamma)\overline{M_c}^{-k}$ for all $\gamma\in \Gamma$. Therefore, $\rho_1$ is conjugate to $\rho_2$ if and only if there exists $k$ such that $u_2\left(\overline{M_c}^k\cdot u_1\right)^{-1}$ belongs to $B^1(\Gamma,\overline{D}_{\overline{\rho}})$.

With this equivalence, we deduce that the bijection $f$ in proposition \ref{Hrho} induces the wanted bijection :

$$\overline{f}:\left|\begin{array}{clc}H^1(\Gamma,\overline{D}_{\overline{\rho}})/\langle \overline{M_c}\rangle&\longrightarrow& \overline{\mathcal{H}_{\overline{\rho}}}\\ \left[u\right]\text{ mod }\langle \overline{M_c}\rangle&\longmapsto& \left[\gamma\mapsto u(\gamma)\overline{\rho}(\gamma)\right]\end{array}\right.\text{.}$$

Remark that $\overline{f}$ makes the following diagram commute :

$$\xymatrix{Z^1(\Gamma,\overline{D}_{\overline{\rho}})\ar[rr]^f\ar[d]&&\mathcal{H}_{\overline{\rho}}\ar[d]\\
H^1(\Gamma,\overline{D}_{\overline{\rho}})/\langle \overline{M_c}\rangle\ar[rr]^{\overline{f}}&& \overline{\mathcal{H}_{\overline{\rho}}}
}$$

Since $f$ is continuous and the projection from $Z^1(\Gamma,\overline{D}_{\overline{\rho}})$ to  $H^1(\Gamma,\overline{D}_{\overline{\rho}})/\langle \overline{M_c}\rangle$ is open, it follows that $\overline{f}$ is continuous. Likewise, since $f^{-1}$ is continuous and the projection of $\mathcal{H}_{\overline{\rho}}$ onto $\overline{\mathcal{H}_{\overline{\rho}}}$ is open, it follows that $\overline{f}^{-1}$ is continuous. \end{proof}

Let $\overline{\rho}$ be a morphism from $\Gamma$ to $\langle \overline{M_c}\rangle$. We will need later a more explicit computation of $H^1(\Gamma,\overline{D}_{\overline{\rho}})$,  when   we fix particular examples of $\Gamma$.

\begin{proposition}\label{Transgr}

Let $\Gamma$ be a finitely generated group, $p$ be a prime number and $\overline{\rho}$ be a non-trivial morphism from $\Gamma$ to $\langle \overline{M_c}\rangle$. Let $\gamma_0\in \Gamma$ verifying $\overline{\rho}(\gamma_0)=\overline{M_c}$. Then the restriction map from $\Gamma$ to $K$ induces an homeomorphism between 
$$H^1(\Gamma,\overline{D}_{\overline{\rho}})\text{ and } \left\lbrace f\in \Hom(\Ker(\overline{\rho}),\overline{D})\left|\begin{array}{l}f(\gamma_0^p)=\overline{I_p}\\f(\gamma_0\gamma\gamma_0^{-1})=\overline{M_c}\cdot f(\gamma)\text{, } \forall \gamma\in K\end{array}\right. \right\rbrace\text{.}$$
\end{proposition}

\begin{proof}

To simplify the proof, we denote $K:=\Ker(\overline{\rho})$. We have an exact sequence of groups $\xymatrix{1\ar[r]& K\ar[r]& \Gamma\ar[r]^{\overline{\rho}}&\langle \overline{M_c}\rangle\ar[r]&1}$. Since $\Gamma$ acts on $\overline{D}$ via $\overline{\rho}$, we may write the first terms of the Inflation-Restriction sequence (see Proposition~\ref{InfRes}). 
$$\xymatrix{H^1(\langle \overline{M_c}\rangle,\overline{D}_{\overline{\rho}}^K)\ar[r]^{Inf}&H^1(\Gamma,\overline{D}_{\overline{\rho}})\ar[r]^{Res}&H^1(K,\overline{D}_{\overline{\rho}})^{\langle \overline{M_c}\rangle}\ar[r]^{T}&H^2(\langle \overline{M_c}\rangle,\overline{D}_{\overline{\rho}}^K) }\text{.}$$

The subgroup $K$ acts trivially on $\overline{D}$ so that $H^1(K,\overline{D}_{\overline{\rho}})=\Hom(K,\overline{D})$ by Lemma \ref{H1triv}.  The Inflation-Restriction sequence becomes :
\begin{equation}
\xymatrix{H^1(\langle \overline{M_c}\rangle,\overline{D})\ar[r]^{Inf}&H^1(\Gamma,\overline{D}_{\overline{\rho}})\ar[r]^{Res}&\Hom(K,\overline{D})^{\langle \overline{M_c}\rangle}\ar[r]^{T}&H^2(\langle \overline{M_c}\rangle,\overline{D}) }\text{.}
\label{InflaRestri}
\end{equation}

Applying Lemma \ref{H1cycl}, with $G$ being $\langle \overline{M_c}\rangle$, $g$ being $\overline{M_c}$ and $M:=\overline{D}$ :
\begin{equation}
H^1(\langle \overline{M_c}\rangle,\overline{D})=\frac{\Ker(\Norme_{\overline{D}})}{\im(\Trace_{\overline{D}})}
\label{H1cyclic}
\end{equation}

Let $b$ be a diagonal matrix with coefficients $b_0,\dots,b_{p-1}$ on its diagonal :
$$b(M_c\cdot b)^{-1}=  \begin{pmatrix}b_{0}b_{p-1}^{-1}&&\\&\ddots&\\&&b_{p-1}b_{p-2}^{-1}\end{pmatrix} \text{.}$$

Let $c\in D$ then define $b_{p-1-i}:=c_{0,0}\cdots c_{i ,i}$. Since $c_{0,0}\cdots c_{p-1,p-1}=\det(c)=1$, an induction on $i$ shows that $b_ib_{i-1}^{-1}=c_{i,i}$.  Multiplying  $b$ by  $\det(b)^{1/p}$ (this does not change the equation), we may even assume that $\det(b)=1$ and then   $\Trace_{\overline{D}}(\overline{b})=\overline{b}(\overline{M_c}\cdot\overline{b})^{-1}=\overline{c}$. 

Therefore,  $\im(\Trace_{\overline{D}})=\overline{D}$.   Equation \ref{H1cyclic} implies that $H^1(\langle \overline{M_c}\rangle,\overline{D})$ is trivial.  The exact sequence \ref{InflaRestri} implies that the restriction morphism $\Res$ is an isomorphism between $H^1(\Gamma,\overline{D}_{\overline{\rho}})$ and the kernel of the Transgression map $T$. Furthermore, the restriction map is clearly continuous and open on its image.

From Appendix \ref{grocoh}, we know that the action of $\langle \overline{M_c}\rangle$ over $\Hom(K,\overline{D})$ used to define $\Hom(K,\overline{D})^{\langle \overline{M_c}\rangle}$ in the Inflation-Restriction sequence is defined by :
$$
\left(\overline{M_c}\cdot f\right)(\gamma):=\overline{M_c}\cdot f(\gamma_0^{-1}\gamma\gamma_0) \text{, for $\gamma\in K$ and $f\in \Hom(K,\overline{D})$.}
$$
 
Finally, Paragraph 10.2 in \cite{D-H-W} contains an explicit formula to compute the Transgression map when the kernel acts trivially (their formula is written additively, here it is written multiplicatively). If $f\in \Hom(K,\overline{D})^{\langle \overline{M_c}\rangle}$ then $T(f)$ is the cohomology class of the following $2$-cocycle on $\langle \overline{M_c}\rangle$ :

\begin{equation}t(f)(\overline{M_c^k},\overline{M_c^l})=\left\lbrace
\begin{array}{cl}
  \overline{I_p}  &  \mbox{if $0\leq k,l\leq p-1$ and $k+l<p$}\\
  f(\gamma_0^p)^{-1}&\mbox{if $0\leq k,l\leq p-1$ and $k+l\geq p$}
\end{array}
\right.
\label{transgexpli}
\end{equation}

Assume that $f$ is in the kernel of the Transgression map then $t(f)$ is a $2$-coboundary in $B^2(\langle \overline{M_c}\rangle,\overline{D}_{\overline{\rho}})$ ; i.e.  there is a map $g:\langle \overline{M_c}\rangle\rightarrow \overline{D}$ such that for all $k,l$ in $\{0,\dots,p-1\}$ : $t(f)(\overline{M_c^k},\overline{M_c^l})=g(\overline{M_c^k})\overline{M_c^k}\cdot g(\overline{M_c^l}) g(\overline{M_c^{l+k}})^{-1}$. In particular, for $k=1$ and $l<p-1$ : $g(\overline{M^{l+1}})=g(\overline{M_c})\overline{M_c}\cdot g(\overline{M_c^l})$. An  induction shows that $g(\overline{M_c^{l}})=g(\overline{M_c})\dots\overline{M_c^{l-1}}\cdot g(\overline{M_c})$. For $k=1$ and $l=p-1$ :
$$f(\gamma_0^p)^{-1}=g(\overline{M_c})\overline{M_c}\cdot g(\overline{M_c^{p-1}})=\underbrace{g(\overline{M_c})\dots\overline{M_c^{p-1}}\cdot g(\overline{M_c})}_{\Norme_{\overline{D}}(g(\overline{M_c}))}=\overline{I_p} \text{.}
$$
Therefore, if $f$ is in $\Ker(T)$ then $f(\gamma_0^p)$ is trivial. Conversely, using Equation \ref{transgexpli}, if $f(\gamma_0^p)$ is trivial then $f$ is  in $\Ker(T)$. Combining this with the explicit definition of the action of $\langle \overline{M_c}\rangle$ on $\Hom(K,\overline{D})$ :

$$\Ker(T)= \left\lbrace f\in \Hom(K,\overline{D})\left|\begin{array}{l}f(\gamma_0^p)=\overline{I_p}\\f(\gamma_0\gamma\gamma_0^{-1})=\overline{M_c}\cdot f(\gamma)\text{, } \forall \gamma\in K\end{array}\right. \right\rbrace $$

Since we have already proved that the restriction morphism is an homeomorphism between $H^1(\Gamma,\overline{D}_{\overline{\rho}})$ and $\Ker(T)$, we are done. \end{proof}

\begin{remark}\label{Hirred} If $\overline{\rho}:\Gamma\to\langle \overline{M_c}\rangle$ is trivial, then any representation in $\mathcal{H}_{\overline{\rho}}$ has its image included in $\overline{D}$ and therefore cannot be irreducible. If $\overline{\rho}:\Gamma\to\langle \overline{M_c}\rangle$ is non-trivial then a representation $\rho$ in $\mathcal{H}_{\overline{\rho}}$ is irreducible if and only if it is not conjugate to $\langle\overline{M_c}\rangle$ (Remark \ref{irred}), if and only if $\Ker(q_{|\rho(\Gamma)})$ is not trivial. It follows that a representation $\rho$ belonging to $\mathcal{H}_{\overline{\rho}}$ is irreducible if and only if its corresponding $1$-cocycle (via the correspondence of proposition \ref{Hrho}) is not a $1$-coboundary. 
\end{remark}

\begin{remark}\label{decomposi} We denote $\mathcal{H}_{\overline{\rho}}^i$ the set of irreducible representations in $\mathcal{H}_{\overline{\rho}}$ and $\overline{\mathcal{H}_{\overline{\rho}}}^i$ their conjugacy classes up to conjugation by $\overline{D}\rtimes \langle \overline{M_c}\rangle$. We have :

\begin{equation}\chi^i_{Sing}(\Gamma,PSL(p,\mathbb{C}))=\bigcup_{\substack{\overline{\rho}\in \Hom(\Gamma,\langle \overline{M_c}\rangle)\\\overline{\rho}\textnormal{ non-trivial}}}\varphi(\overline{\mathcal{H}_{\overline{\rho}}}^{i})
\label{decompbadlocus}
\end{equation}
\end{remark}

In the next section we will prove that $\varphi$ restricted to $\overline{\mathcal{H}_{\overline{\rho}}}^{i}$ is an homeomorphism onto its image. Therefore,  Propositions \ref{barHrho} and \ref{Transgr} will eventually give a topological understanding of the singular locus. 

\section{A domain of injectivity for $\varphi$}\label{secinj}

\homeobarHrho*

We first need to prove :

\begin{proposition}\label{injectivity}

Let $\Gamma$ be a finitely generated group,  $p$ be a prime number and $\overline{\rho}$ be a non-trivial morphism in $\Hom(\Gamma,\langle \overline{M_c}\rangle)$ and $\rho,\rho'$ be two irreducible representations in $\mathcal{H}_{\overline{\rho}}$. If there exists $g\in PSL(p,\mathbb{C})$ such that $g\cdot \rho=\rho'$ then $g\in \overline{D}\rtimes\langle\overline{M_c}\rangle$. 

\end{proposition}

\begin{proof}

Let $\gamma_0\in \Gamma$ such that $\overline{\rho}(\gamma_0)=M_c$. Since  $\rho,\rho'$ are in $\mathcal{H}_{\overline{\rho}}$ there exists $d_0,d_0',\dots,d_{p-1},d_{p-1}'$ such that :

$$\rho(\gamma_0)=\overline{\begin{pmatrix}d_0&&\\&\ddots&\\&&d_{p-1}\end{pmatrix}}\overline{M_c}\textnormal{ and } \rho'(\gamma_0)=\overline{\begin{pmatrix}d_0'&&\\&\ddots&\\&&d_{p-1}'\end{pmatrix}}\overline{M_c}\text{.} $$

Like in proposition \ref{centralconj}, there are  $g_1$ and $g_2\in \overline{D}$ verifying   $g_1\rho(\gamma_0)g_1^{-1}=\overline{M_c}$ and $g_2\rho'(\gamma_0)g_2^{-1}=\overline{M_c}$. Since $g\cdot \rho=\rho'$,   $\overline{M_c}=(g_1^{-1}g_2g)\overline{M_c}(g_1^{-1}g_2g)^{-1}$. 
 
 Let us denote $\overline{h}:=g_1^{-1}g_2g$. We have just seen that $\overline{h}$ centralizes $\overline{M_c}$. Lemma \ref{centrMc}  implies that there are  $a_0,\dots, a_{p-1}\in\mathbb{C}$ and an integer $s\geq 0$ such that :
$$h=\sum_{k=0}^{p-1}a_kM_c^kD(\xi)^s\text{.}$$
Since $g_1\cdot\rho$ is irreducible, there is $\gamma\in \Gamma$ such that  $g_1\cdot\rho(\gamma)$  is diagonal and not trivial. Remark that $\gamma$ belongs to $\Ker(\overline{\rho})$ and therefore $g_2\cdot\rho'(\gamma)$ is  diagonal as well. Let $h_1,h_2\in D$ verifying  $\overline{h_1}=g_1\cdot\rho(\gamma)$ and $\overline{h_2}=g_2\cdot\rho(\gamma)$. By definition of $h$, there is $l\geq 0$ verifying  $hh_1h^{-1}= \xi^lh_2$. Hence :
$$\sum_{k=0}^{p-1}a_kM_c^kD(\xi)^sh_1=
\xi^l\sum_{k=0}^{p-1}a_kh_2M_c^kD(\xi)^s\text{.}$$

Applying Lemma \ref{freeMc},  we have $a_kM_c^kD(\xi)^sh_1=a_kh_2M_c^kD(\xi)^s$,  for $0\leq k\leq p-1$. Since $D(\xi)^s$ commutes with $h_1$ : $a_k(M_c^kh_1-\xi^lh_2M_c^k)=0$  holds for $0\leq k\leq p-1$. Assume there are $0\leq k_1<k_2\leq p-1$ verifying  $a_{k_1}\neq 0$ and $a_{k_2}\neq 0$. Then $h_1=\xi^lM_c^{-k_1}h_2M_c^{k_1}$ and $\xi^lh_2=M_c^{k_2}h_1M_c^{-k_2}$ and therefore :
$$h_1=M_c^{k_2-k_1}h_1M_c^{k_1-k_2}\text{.} $$

Since $M_c^{k_2-k_1}=M_{c^{k_2-k_1}}$ and $c^{k_2-k_1}$ is a cyclic permutation of order $p$, the equation above implies that $h_1=g_1\cdot\rho(\gamma)$ is trivial, which is a contradiction.   As a result, there  is a unique $k$ such that $a_k\neq 0$ and   $\overline{h}$ belongs to  $\overline{D}\rtimes\langle\overline{M_c} \rangle$.  Since $g=g_2^{-1}g_1\overline{h}$, $g$ also belongs to  $\overline{D}\rtimes\langle\overline{M_c} \rangle$. \end{proof}

To finish the proof of Theorem \ref{homeobarHrho}, we need to deal with the topology :
\begin{proof}[Theorem \ref{homeobarHrho}]

Proposition \ref{injectivity} proves that $\varphi_{|\overline{\mathcal{H}_{\overline{\rho}}}^i}$  is injective.  We recall the diagram 

$$\xymatrix{\Hom^i(\Gamma,\overline{D}\rtimes \langle \overline{M_c}\rangle) \ar[rr]^{\iota}\ar[d]^{\textnormal{mod }\overline{D}\rtimes \langle \overline{M_c}\rangle=\psi_1}&&\Hom^i(\Gamma,PSL(p,\mathbb{C})) \ar[d]^{\textnormal{mod }PSL(p,\mathbb{C})=\psi_2}\\ \chi^i(\Gamma,\overline{D}\rtimes \langle \overline{M_c}\rangle)\ar[rr]^{\varphi} &&\chi^i(\Gamma,PSL(p,\mathbb{C}))}$$

We will denote  $\psi_1$ the projection mod $\overline{D}\rtimes \langle \overline{M_c}\rangle$ and   $\psi_2$ the projection mod $PSL(p,\mathbb{C})$. Since they are  projections by a topological group action, $\psi_1$ and $\psi_2$ are both continuous and open. Since $\iota$ is induced by the inclusion of $\overline{D}\rtimes \langle \overline{M_c}\rangle$ into $PSL(p,\mathbb{C})$ which is a homeomorphism onto its image, $\iota$ is also a homeomorphism onto its image and $\im(\iota)$ is a closed subset of $PSL(p,\mathbb{C})$.    Since $\psi_2$ is continuous and $\psi_1$ is open, $\varphi$ is continuous and, in particular,  $\varphi_{|\overline{\mathcal{H}_{\overline{\rho}}}^i}$ is continuous. Our next order   of business is to show that $\varphi_{|\overline{\mathcal{H}_{\overline{\rho}}}^i}$ is open in its image.  We restrict the diagram to  $\mathcal{H}_{\overline{\rho}}^i$ :

$$\xymatrix{
\mathcal{H}_{\overline{\rho}}^i\ar[rrrr]^{\iota_{|\mathcal{H}_{\overline{\rho}}^i}} \ar[d]_{\psi_{1|\mathcal{H}_{\overline{\rho}}^i}}  &&&&\Hom^i(\Gamma,PSL(p,\mathbb{C})) \ar[d]_{\psi_2}\\  \overline{\mathcal{H}_{\overline{\rho}}}^i\ar[rrrr]^{\varphi_{|\overline{\mathcal{H}_{\overline{\rho}}}^i}} &&&&\chi^i(\Gamma,PSL(p,\mathbb{C}))
    }$$

Since each $\mathcal{H}_{\overline{\rho}}^i$ is closed and $\Hom^i(\Gamma,\overline{D}\rtimes \langle \overline{M_c}\rangle)$ is a finite disjoint union of such $\mathcal{H}_{\overline{\rho}}^i$'s, each $\mathcal{H}_{\overline{\rho}}^i$ is open in $\Hom^i(\Gamma,\overline{D}\rtimes \langle \overline{M_c}\rangle)$. It follows that $\psi_1$ restricted to $\mathcal{H}_{\overline{\rho}}^i$ is still continuous and open and $\iota$ restricted to $\mathcal{H}_{\overline{\rho}}^i$ is still an homeomorphism onto its image. The group  $PSL(p,\mathbb{C})$ acts properly on $\Hom^i(\Gamma,PSL(p,\mathbb{C}))$ (Proposition 1.1 in \cite{J-M}), i.e. the following function is proper :

\begin{displaymath}
\zeta :
\left|
  \begin{array}{rcl}
PSL(p,\mathbb{C})\times \Hom^i(\Gamma,PSL(p,\mathbb{C}))&\longrightarrow& \Hom^i(\Gamma,PSL(p,\mathbb{C}))^2\\
(g,\rho)&\longmapsto&  (g\cdot \rho,\rho)  \\
  \end{array}
\right.\text{.}
\end{displaymath}

Since both $PSL(p,\mathbb{C})$ and $\Hom^i(\Gamma,PSL(p,\mathbb{C}))$ are locally compact and Hausdorff,  the function $\zeta$ is   closed. For any   closed set $F$ in $\Hom^i(\Gamma,PSL(p,\mathbb{C}))$, the set $\zeta(PSL(p,\mathbb{C})\times F)$ is closed. Therefore,  its projection  on the first coordinate $PSL(p,\mathbb{C})\cdot F$ is closed in  $\Hom^i(\Gamma,PSL(p,\mathbb{C}))$.

Let $U$  be an open subset of $\overline{\mathcal{H}_{\overline{\rho}}}^i$. We want to prove that $\varphi(U)$ is open in $\varphi(\overline{\mathcal{H}_{\overline{\rho}}}^i)$. Denote $U_0:=(\psi_{1|\mathcal{H}_{\overline{\rho}}^i})^{-1}(U)$ which is open by continuity of $\psi_1$. Then  $\iota(U_0)$ is open in $\iota(\mathcal{H}_{\overline{\rho}}^i)$, hence there exists an open set $V$ in $\Hom^i(\Gamma,PSL(p,\mathbb{C}))$ such that $V\cap \iota(\mathcal{H}_{\overline{\rho}}^i)=\iota(U_0) $.

Let $F:=V^c\cap \iota(\mathcal{H}_{\overline{\rho}}^i)$. Since $V^c$ is closed and $\iota(\mathcal{H}_{\overline{\rho}}^i)$ is closed,   $F$ is closed, and using properness as explained above, its saturation  $PSL(p,\mathbb{C})\cdot F$ is   closed. Let $V_0$  be $V-PSL(p,\mathbb{C})\cdot F$ then $V_0$ is open. 

\bigskip

By definition, $V_0\cap\iota(\mathcal{H}_{\overline{\rho}}^i)$ is contained in  $V\cap \iota(\mathcal{H}_{\overline{\rho}}^i)$.  Conversely, let $\rho$ be in $\iota(U_0)= V\cap \iota(\mathcal{H}_{\overline{\rho}}^i)$, we want to show that $\rho$ is not conjugate to an element of $F$ (this will imply that $\rho$ belongs to $V_0\cap\iota(\mathcal{H}_{\overline{\rho}}^i)$) and we do it by contradiction. Assume that $\rho\in PSL(p,\mathbb{C})\cdot F$, i.e. there is $g\in PSL(p,\mathbb{C})$ such that $g\cdot \rho\in F$. Since $\rho$ and $g\cdot \rho$ are both irreducible and valued in $ \overline{D}\rtimes\langle\overline{M_c}\rangle$ by definition, Proposition \ref{injectivity} implies that $g$ belongs to $\overline{D}\rtimes \langle \overline{M_c}\rangle$. Since $\rho$ belongs to $\iota(U_0)$ and $g\cdot\rho$ does not,   $U_0$ is not stable by the action of $ \overline{D}\rtimes\langle\overline{M_c}\rangle$, which is in contradiction with the very definition of $U_0$. Therefore $V_0\cap\iota(\mathcal{H}_{\overline{\rho}}^i)=V\cap \iota(\mathcal{H}_{\overline{\rho}}^i)$. Since $\psi_2\circ\iota=\varphi\circ\psi_1$, we have  $\varphi(U)=\psi_2(\iota(U_0))=\psi_2(V\cap  \iota(\mathcal{H}_{\overline{\rho}}^i))$. Therefore $\varphi(U)=\psi_2(V_0\cap  \iota(\mathcal{H}_{\overline{\rho}}^i))$. 

\bigskip

 It remains to show that $\psi_2(V_0\cap  \iota(\mathcal{H}_{\overline{\rho}}^i))=\psi_2(V_0)\cap \psi_2( \iota(\mathcal{H}_{\overline{\rho}}^i))$. The left-hand side is clearly contained in the right-hand side. Since $V_0$ is stable (by definition) by the action of $PSL(p,\mathbb{C})$, we have the other inclusion and therefore, we have the equality. 

\bigskip

Finally  $\varphi(U)$ is the intersection of $\psi_2(V_0)$ which is open (since $\psi_2$ is open)  and $\psi_2( \iota(\mathcal{H}_{\overline{\rho}}^i))=\varphi(\overline{\mathcal{H}_{\overline{\rho}}}^i)$. Therefore $\varphi(U)$ is open in $\varphi(\overline{\mathcal{H}_{\overline{\rho}}}^i)$. The function $\varphi_{\overline{\mathcal{H}_{\overline{\rho}}}^i}$ is continuous and open in its image, since  it is also injective, it is an homeomorphism onto its image.  \end{proof}

As a result, the singular locus of a character variety in $PSL(p,\mathbb{C})$ is a finite union of topological spaces, namely the $\varphi(\overline{\mathcal{H}_{\overline{\rho}}}^i)$'s,  whose topology is given by Propositions \ref{barHrho} and \ref{Transgr}. Before studying the intersections between these spaces, we remark that some of them may be equal to each other.

\begin{lemma}\label{identifvarphiHrho}

Let $\Gamma$ be a finitely generated group, $p$ be a prime number and $\overline{\rho},\overline{\rho}'$ be two non-trivial elements in $\Hom(\Gamma,\langle\overline{M_c}\rangle)$. Then,   $\Ker(\overline{\rho})=\Ker(\overline{\rho}')$ if and only if there is $\phi\in\Aut(\langle\overline{M_c}\rangle)$ verifying       $\overline{\rho}'=\phi\circ\overline{\rho}$. When it is true, $\varphi(\overline{\mathcal{H}_{\overline{\rho}}}^i)=\varphi(\overline{\mathcal{H}_{\overline{\rho}'}}^i)$.

\end{lemma}

\begin{proof}

The equivalence in this lemma is straightforward.
 
 Let $\phi\in\Aut(\langle\overline{M_c}\rangle)$, then $\phi$ is uniquely determined by $l\in (\mathbb{Z}/p)^*$ where $\phi(\overline{M_c})=\overline{M_c}^l$.  We defined $c$ to be the permutation of $\mathbb{Z}/p$ sending $i$ to $i+1$. Let us now define  the permutation $\sigma_l$ of $\mathbb{Z}/p$ sending $i$ to $l\times i$. We easily see that $\sigma_l\circ c(i)=l\times i+l=c^l\circ \sigma_l(i)$. Therefore  $\sigma_lc\sigma_l^{-1}=c^l $ and $M_{\sigma_l}M_{c}M_{\sigma_l}^{-1}=M_c^l $. Let $M:=M_{\sigma_l}$ if $\det(M_{\sigma_l})=1$ and $-M_{\sigma_l}$ if $\det(M_{\sigma_l})=-1$. It follows that the conjugation by  $\overline{M}$ sends  $\iota(\mathcal{H}_{\overline{\rho}}^i)$ to $\iota(\mathcal{H}_{\phi\circ \overline{\rho}}^i)$. As a result $\varphi(\overline{\mathcal{H}_{\overline{\rho}}}^i)=\varphi(\overline{\mathcal{H}_{\phi\circ \overline{\rho}}}^i)$.  \end{proof}

\begin{remark}\label{newdecomposi} Fix a finitely generated group $\Gamma$ and  a prime number $p$. Then,  for any  normal subgroup $K$ of index $p$ in $\Gamma$, fix one group morphism $\overline{\rho}_K$ in $\Hom(\Gamma,\langle\overline{M_c}\rangle)$ verifying $K=\Ker(\overline{\rho})$. Lemma \ref{identifvarphiHrho} and the decomposition of the singular locus in Equation \ref{decompbadlocus} implies the following decomposition :

\begin{equation}\chi^i_{Sing}(\Gamma,PSL(p,\mathbb{C}))=\bigcup_{\substack{K\triangleleft \Gamma\\   \text{$[\Gamma:K]=p$} } }\varphi(\overline{\mathcal{H}_{\overline{\rho}_K}}^i)\text{.}
\label{decompbadlocus2}
\end{equation}

We shall call $\varphi(\overline{\mathcal{H}_{\overline{\rho}_K}}^i)$ the  \textit{pseudo-component} associated to $K$. When  $\Gamma$ is a free group, each of these pseudo-components is an irreducible component of $\chi^i_{Sing}(\Gamma,PSL(p,\mathbb{C}))$ (Remark \ref{irredcompfree}). However, in general it has no reason to be an irreducible component of the singular locus (the surface group is a counter example, see Remark \ref{irredcompsurf}).

\end{remark}

\section{The intersection pattern in the singular locus}\label{secint}

In this section, we give a description of the intersection pattern between the different pseudo-components defined above. The next lemma is straightforward and will be used several times throughout this section.

\begin{lemma}\label{notprop}
Let $\Gamma$ be a  group, $p$ be a prime number and $K,K'$ be two different normal subgroups of index $p$ in $\Gamma$. Then   $\Gamma/K\cap K'$ is isomorphic to $(\mathbb{Z}/p)^2$.
\end{lemma}

\begin{proposition}\label{abelirred}
Let $\Gamma$ be a finitely generated group, $p$ be a prime number and $K,K'$ be two different normal subgroups of index $p$ in $\Gamma$. Let $\rho:\Gamma\to PSL(p,\mathbb{C})$ be a representation whose conjugacy class belongs to $\varphi(\overline{\mathcal{H}_{\overline{\rho}_K}}^i)\cap \varphi(\overline{\mathcal{H}_{\overline{\rho}_{K'}}}^i)$. Then $\Ker(\rho)=K\cap K'$ and $\rho(\Gamma)$ is conjugate to $\langle \overline{D(\xi)}\rangle \times\langle\overline{M_c}\rangle$. 

\end{proposition}
\begin{proof}By definition, $\rho$ is both conjugate to a representation in $\iota(\overline{\mathcal{H}_{\overline{\rho}_K}}^i)$ and to a representation in $\iota(\overline{\mathcal{H}_{\overline{\rho}_{K'}}}^i)$. Let $g,h\in PSL(p,\mathbb{C})$ verifying that $g\rho g^{-1}$ (resp. $h\rho h^{-1}$) belongs to $\iota(\overline{\mathcal{H}_{\overline{\rho}_K}}^i)$ (resp. $\iota(\overline{\mathcal{H}_{\overline{\rho}_{K'}}}^i)$). This implies that $g\rho(K) g^{-1}\leq \overline{D}$ and $h\rho(K') h^{-1}\leq \overline{D}$. In  particular,  $\rho(K)$ and $\rho(K')$ are both abelian.

\bigskip

Since $K$ and $K'$ are different and both of index $p$ in $\Gamma$, there is $\gamma_1\in K'\cap K^c$  verifying $g\rho(\gamma_1)g^{-1}=\overline{d_1}\text{ }\overline{M_c}$ with $\overline{d_1}\in\overline{D}$.  Let $\gamma\in K\cap K'$, since  $g\rho(K')g^{-1}$ is abelian, $g\rho(\gamma_1)g^{-1} $ and $g\rho(\gamma)g^{-1}$ commute. Since $g\rho(\gamma)g^{-1}\in\overline{D}$, it also commutes with $\overline{d_1}$ and, therefore, it commutes with $\overline{M_c}$.  By Lemma \ref{centrMc}, an element in $\overline{D}$ commuting with $\overline{M_c}$ necessarily belongs to $\langle \overline{D(\xi)}\rangle$. Thus,    $g\rho(K\cap K')g^{-1}$ is included in $\langle \overline{D(\xi)}\rangle$. Likewise, $h\rho(K\cap K')h^{-1}$ is included in $\langle \overline{D(\xi)}\rangle$.

\bigskip

Assume  $K\cap K'$ is not included in $\Ker(\rho)$. Then,   the groups $g\rho(K\cap K')g^{-1}$ and $h\rho(K\cap K')h^{-1}$ are both equal to $\langle\overline{D(\xi)}\rangle$. Thus, $hg^{-1}$ normalizes $\langle\overline{D(\xi)}\rangle$.

Let $0<k<p$ be an integer such that $hg^{-1}\overline{D(\xi)}gh^{-1}=\overline{D(\xi)}^k$ and $M_{\sigma}$ be a permutation matrix such that $M_{\sigma}D(\xi)^kM_{\sigma}^{-1}=D(\xi)$. Then $\overline{M_{\sigma}}hg^{-1}$ commutes with $\overline{D(\xi)}$, whence (Lemma \ref{centrdxi})  belongs to $\overline{D}\rtimes\langle\overline{M_c}\rangle$. Therefore, $hg^{-1}$ normalizes $\overline{D}$ and, in particular $K=K'$ which is a contradiction. Therefore $\Ker(\rho)=K\cap K'$. 

\bigskip

Lemma \ref{notprop} states that $\Gamma/K\cap K'$ is isomorphic to $(\mathbb{Z}/p)^2$. Therefore, $\rho(\Gamma)$ is  abelian irreducible.  Theorem \ref{classifcentr} implies that any abelian irreducible subgroup of $PSL(p,\mathbb{C})$ is conjugate to    $\langle \overline{D(\xi)}\rangle \times\langle\overline{M_c}\rangle$, whence the result. \end{proof}

\begin{remark}\label{paraminter} Assume $\Gamma$, $p$, $K$ and $K'$ are given like in the preceding proposition. This proposition implies that, up to conjugation, any representation $\rho$ whose conjugacy class belongs to $\varphi(\overline{\mathcal{H}_{\overline{\rho}_K}}^i)\cap \varphi(\overline{\mathcal{H}_{\overline{\rho}_{K'}}}^i)$ factorizes through the projection to $\Gamma/K\cap K'$ in an isomorphism between $\Gamma/K\cap K'$ and $\langle \overline{D(\xi)}\rangle \times\langle\overline{M_c}\rangle$. Therefore, to count the number of elements in $\varphi(\overline{\mathcal{H}_{\overline{\rho}_K}}^i)\cap \varphi(\overline{\mathcal{H}_{\overline{\rho}_{K'}}}^i)$, it suffices to count the number of isomorphisms between $\Gamma/K\cap K'$ and $\langle \overline{D(\xi)}\rangle \times\langle\overline{M_c}\rangle$ up to conjugation in $PSL(p,\mathbb{C})$.

 \end{remark}

Let $p$ be a prime number. The group $\langle \overline{D(\xi)}\rangle \times\langle\overline{M_c}\rangle$ is a $\mathbb{Z}/p$-vector space of dimension $2$  and has the following  $\mathbb{Z}/p$-basis $(  \overline{D(\xi)},\overline{M_c})$. Once the basis is fixed, it is natural to identify its automorphism group to $GL(2,\mathbb{Z}/p)$.

\begin{lemma}\label{countZp2}
Let $p$ be a prime number. Then, the subgroup of  $\Aut(\langle \overline{D(\xi)}\rangle \times\langle\overline{M_c}\rangle)$ induced by the normalizer of  $\langle \overline{D(\xi)}\rangle \times\langle\overline{M_c}\rangle$ in $PSL(p,\mathbb{C})$ is  $SL(2,\mathbb{Z}/p)$. 

\end{lemma}

\begin{proof} Assume $p$  odd. We denote  $\Aut_0(\langle \overline{D(\xi)}\rangle \times\langle\overline{M_c}\rangle)$ the subgroup of automorphisms of $\langle \overline{D(\xi)}\rangle \times\langle\overline{M_c}\rangle$ which can be realized as the conjugation by an element of $PSL(p,\mathbb{C})$.

Let $V_1:=\lambda V_0$ where $V_0:=(\xi^{ij})_{0\leq i,j\leq p-1}$ is the Vandermonde matrix and $\lambda$ be $\det(V_0)^{-\frac{1}{p}}=(-1)^{\frac{p-1}{2p}}/p$. Direct computations lead to $V_0D(\xi)V_0^{-1}=M_c^{-1}$ and $V_0M_cV_0^{-1}=D(\xi)$. Therefore  $\overline{V_1}\text{ }\overline{D(\xi)}\text{ }\overline{V_1}^{-1}=\overline{M_c}^{-1}$ and $ \overline{V_1}\text{ }\overline{M_c}\text{ }\overline{V_1}^{-1}=\overline{D_{\xi}}$. Thus $\overline{V_1}$ normalizes $\langle \overline{D(\xi)}\rangle\times \langle \overline{M_c}\rangle$ and induces by conjugation the automorphism $\begin{pmatrix}0&1\\ -1&0\end{pmatrix}$ on $\langle \overline{D(\xi)}\rangle \times\langle\overline{M_c}\rangle$. 

Let $S$ be a diagonal matrix where $S_{i,i}=\xi^{\frac{i(i+1)}{2}}$ for $i=0,\dots,p-1$. Then $\det(S)=\xi^{\frac{(p-1)p(p+1)}{6}}=1$ since $p$ is odd. Since $S$ is diagonal, it commutes with $D(\xi)$ therefore $\overline{S}\text{ }\overline{D(\xi)}\text{ }\overline{S}^{-1}=\overline{D(\xi)}$. Direct computations show that $SM_cS^{-1}=D(\xi)M_c$. Therefore  $\overline{S}\text{ }\overline{M_c}\text{ }\overline{S}^{-1}=\overline{M_c}\text{ }\overline{D(\xi)}$. Whence $\overline{S}$ normalizes $\langle \overline{D(\xi)}\rangle\times \langle \overline{M_c}\rangle$ and induces by conjugation the automorphism $\begin{pmatrix}1&1\\  0&1\end{pmatrix}$ on $\langle \overline{D(\xi)}\rangle \times\langle\overline{M_c}\rangle$.

Therefore, the subgroup generated by $\begin{pmatrix}0&1\\ -1&0\end{pmatrix}$ and $\begin{pmatrix}1&1\\ 0&1\end{pmatrix}$ is contained in $\Aut_0(\langle \overline{D(\xi)}\rangle \times\langle\overline{M_c}\rangle)$ . Since these matrices generate $SL(2,\mathbb{Z}/p)$, this group is contained in $\Aut_0(\langle \overline{D(\xi)}\rangle \times\langle\overline{M_c}\rangle)$.

\bigskip

Conversely, let  $\psi$ be in $\Aut_0(\langle \overline{D(\xi)}\rangle \times\langle\overline{M_c}\rangle)$. We denote $d$ its determinant. We denote $\psi_d:=\begin{pmatrix}1&0\\ 0&d\end{pmatrix}$. Since $\det(\psi_d^{-1}\psi)=1$, it follows that the automorphism $\psi_d$ also belongs to $\Aut_0(\langle \overline{D(\xi)}\rangle \times\langle\overline{M_c}\rangle)$.

Let $g$ be in $PSL(p,\mathbb{C})$ verifying for all $x\in \langle \overline{D(\xi)}\rangle \times\langle\overline{M_c}\rangle$ that $gxg^{-1}=\psi_d(x)$. Then $g$ commutes with $\overline{D(\xi)}$. Furthermore $g\overline{M_c}g^{-1}=\overline{M_c}^d$. Since Lemma \ref{centrdxi} implies that $g$ belongs to $\overline{D}\rtimes \langle\overline{M_c}\rangle$, we may project this equality onto $ \langle\overline{M_c}\rangle$ to get that $\overline{M_c}=\overline{M_c}^d$ and $d=1$. As a result $\det(\psi)=1$ and we are done.  \end{proof}

Therefore, we can count the number of elements in  the intersection of two pseudo-components  :

\begin{corollary}\label{countinter}

Let $\Gamma$ be a finitely generated group, $p$ be a prime number and $K,K'$ be two different normal subgroups of index $p$ in $\Gamma$. Then $\varphi(\overline{\mathcal{H}_{\overline{\rho}_K}}^i)\cap \varphi(\overline{\mathcal{H}_{\overline{\rho}_{K'}}}^i)$ is finite of cardinal $p-1$.

\end{corollary}

\begin{proof} Following Remark \ref{paraminter}, there are as many elements in $\varphi(\overline{\mathcal{H}_{\overline{\rho}_K}}^i)\cap \varphi(\overline{\mathcal{H}_{\overline{\rho}_{K'}}}^i)$ as there are isomorphisms between $\Gamma/K\cap K'$ and $\langle \overline{D(\xi)}\rangle \times\langle\overline{M_c}\rangle$ up to conjugation.

We fix $\psi_0$ such isomorphism. Since the natural action of $\Aut(\langle \overline{D(\xi)}\rangle \times\langle\overline{M_c}\rangle)$ on the set of isomorphisms between $\Gamma/K\cap K'$ and $\langle \overline{D(\xi)}\rangle \times\langle\overline{M_c}\rangle$ is transitive, it suffices to understand under which condition $\psi\circ\phi_0$ is conjugate to $\phi_0$.  Lemma \ref{countZp2} implies that $\psi\circ\phi_0$ is conjugate to $\phi_0$ if and only if $\psi\in SL(2,\mathbb{Z}/p)$.  Therefore, there are $|GL(2,\mathbb{Z}/p)|/|SL(2,\mathbb{Z}/p)|=p-1$  elements in $\varphi(\overline{\mathcal{H}_{\overline{\rho}_K}}^i)\cap \varphi(\overline{\mathcal{H}_{\overline{\rho}_{K'}}}^i)$.  \end{proof}

\begin{proposition}\label{abirredK}

Let $\Gamma$ be a finitely generated group, $p$ be a prime number and $\rho$ be an abelian irreducible representation from $\Gamma$ to $PSL(p,\mathbb{C})$. Then  the conjugacy class of $\rho$ belongs to exactly $p+1$ pseudo-components.

\end{proposition}

\begin{proof}Let $K$ be a normal subgroup of index $p$ in $\Gamma$. We first show that  the conjugacy class of $\rho$ belongs to $\varphi(\overline{\mathcal{H}_{\overline{\rho}_K}}^i)$ if and only if $\Ker(\rho)$ is contained in $K$. 
 
If some conjugate $g\rho g^{-1}$ of $\rho$ belongs to $\iota(\mathcal{H}_{\overline{\rho}_K}^i)$, then  $\Ker(\rho)$ will  be contained in $\Ker\left(q\circ\left(g\rho g^{-1}\right)\right)=K$.

Assume   that $\Ker(\rho)$ is contained in $K$. Since $\rho$ is abelian irreducible, Theorem \ref{classifcentr} implies that $\rho$ is conjugate to a representation $g\rho g^{-1}$ from $\Gamma$ onto $\langle \overline{D(\xi)}\rangle \times\langle\overline{M_c}\rangle$. 

Since $K$ strictly contains $\Ker(\rho)$, it follows that $g\rho (K)g^{-1}$ is a subgroup of index $p$ in $\langle \overline{D(\xi)}\rangle \times\langle\overline{M_c}\rangle$. The group $SL(2,\mathbb{Z}/p)$ acts transitively by automorphism on the subgroups of index $p$  in $\langle \overline{D(\xi)}\rangle \times\langle\overline{M_c}\rangle$.

Therefore, Lemma \ref{countZp2} implies that there is $h\in PSL(p,\mathbb{C})$ normalizing the group $\langle \overline{D(\xi)}\rangle \times\langle\overline{M_c}\rangle$ and  verifying that $hg\rho(K)g^{-1}h^{-1}=\langle \overline{D(\xi)}\rangle $.

Thus the representation $hg\rho(hg)^{-1}$ is still a representation in $\overline{D}\rtimes\langle\overline{M_c}\rangle$ and  $\Ker\left(q\circ \left(hg\rho(hg)^{-1}\right)\right)=K$. Therefore, its conjugacy class belongs to $\varphi(\overline{\mathcal{H}_{\overline{\rho}_K}}^i)$.

\bigskip

Using this equivalence, there are as many pseudo-components containing the conjugacy class of $\rho$   as  there are normal subgroups of index $p$ containing $\Ker(\rho)$. This is the number of subgroups of index $p$ in $\Gamma/\Ker(\rho)$ which is isomorphic to $(\mathbb{Z}/p)^2$. Thus, there are $p+1$ of them.   \end{proof}

We sum up this combinatorial information in one single theorem. Before that, we recall that if $\Gamma$ is a group then $\Gamma^{Ab}/p\Gamma^{Ab}$ is a $\mathbb{Z}/p$-vector space. Its dimension is called the \textit{$p$-rank} of $\Gamma$. 

\combi*

\begin{proof}\begin{enumerate}
\item In Remark \ref{newdecomposi}, we justified that $\chi^i_{Sing}(\Gamma,PSL(p,\mathbb{C}))$ was the union of  pseudo-components associated to $K$ for $K$ normal subgroups of index $p$ in $\Gamma$. Since the set of such normal subgroups is in bijection with the set of subgroups of index $p$ in $\Gamma^{Ab}/p\Gamma^{Ab}$, we have $\frac{p^r-1}{p-1}$ of them.

\item It is the statement of Proposition \ref{abelirred} and Corollary \ref{countinter}.

\item It is the statement of Proposition \ref{abirredK}.

\item On one hand, for each unordered pair $\{K,K'\}$ of different normal subgroups of index $p$ in $\Gamma$, there are $p-1$ conjugacy classes of abelian irreducible representations from $\Gamma$ to $PSL(p,\mathbb{C})$ in $\varphi(\overline{\mathcal{H}_{\overline{\rho}_K}}^i)\cap \varphi(\overline{\mathcal{H}_{\overline{\rho}_{K'}}}^i)$ and there are $\begin{pmatrix}\frac{p^r-1}{p-1}\\2\end{pmatrix}$ such unordered pairs.

On the other hand, for each conjugacy class abelian irreducible representation from $\Gamma$ to $PSL(p,\mathbb{C})$, there are $\begin{pmatrix}p+1\\2\end{pmatrix}$ unordered pairs $\{K,K'\}$ of different normal subgroups of index $p$ in $\Gamma$ for which the conjugacy class of this representation is contained in $\varphi(\overline{\mathcal{H}_{\overline{\rho}_K}}^i)\cap \varphi(\overline{\mathcal{H}_{\overline{\rho}_{K'}}}^i)$.

Let $N_{\Gamma,p}$ be the number of conjugacy classes of abelian irreducible representations. Then :

$$N_{\Gamma,p}\begin{pmatrix}p+1\\2\end{pmatrix}=\begin{pmatrix}\frac{p^r-1}{p-1}\\2\end{pmatrix}(p-1)\text{, } N_{\Gamma,p}\frac{(p+1)p}{2}=\frac{(p^r-1)(p^r-p)}{2(p-1)^2}(p-1) $$

Therefore, the number of conjugacy classes of abelian irreducible representations from $\Gamma$ to $PSL(p,\mathbb{C})$ is $\frac{(p^r-1)(p^r-p)}{(p-1)p(p+1)}=\frac{(p^r-1)(p^{r-1}-1)}{p^2-1}$.\end{enumerate}  \end{proof}

In this theorem, the combinatorial information    about the singular locus of the character variety only depends on the $p$-rank of $\Gamma$. To understand its topology, Section \ref{diagbyfin} tells us that we need to compute cohomology groups. This is what we will do for free groups in Section \ref{freegrpcase} and for surface groups in Section \ref{surfgrpcase}.

\section{The free group case}\label{freegrpcase}

Define $\Gamma:=\mathbb{F}_l$ to be the free group of rank $l\geq 2$. We recall that the abelianization of $\mathbb{F}_l$ is $\mathbb{Z}^l$. Furthermore the canonical map from $\Aut(\mathbb{F}_l)$ to $\Aut((\mathbb{F}_l)^{Ab})=GL(l,\mathbb{Z})$ is known to be surjective (c.f. \cite{L-S}). Since the set of  normal subgroups $K$ of index $p$ is in bijective correspondence with the set of subgroups of index $p$ in $(\mathbb{F}_l)^{Ab}=\mathbb{Z}^l$, it follows that $\Aut(\mathbb{F}_l)$ acts transitively on the set of normal subgroups of index $p$ in $\Gamma$. Using this, we have :

\begin{proposition}\label{onesinglphiKfree}

Let $l\geq 2$ and $K,K'$ be two normal subgroups of index $p$ in $\mathbb{F}_l$. Then $\varphi(\overline{\mathcal{H}_{\overline{\rho}_K}}^i)$ and $\varphi(\overline{\mathcal{H}_{\overline{\rho}_{K'}}}^i)$ are homeomorphic to each other. 

\end{proposition}

\begin{proof}

Let $\phi$ be an element in $\Aut(\mathbb{F}_l)$ such that $\phi(K)=K'$. Then the precomposition by $\phi^{-1}$ provides a homeomorphism from $\iota(\mathcal{H}_{\overline{\rho}_K})$ to $\iota(\mathcal{H}_{\overline{\rho}_{K'}})$ whose continuous inverse is the precomposition by $\phi$. This homeomorphism induces, on the character variety,  an homeomorphism between $\varphi(\overline{\mathcal{H}_{\overline{\rho}_K}}^i)$ and $\varphi(\overline{\mathcal{H}_{\overline{\rho}_{K'}}}^i)$.
   \end{proof}

\begin{remark}\label{freegrponeK}
This proposition implies that we only need to focus on one single pseudo-component. Here we are going to explain, how to construct a particular normal subgroup of index $p$ in $\mathbb{F}_l$ using topological coverings. What we are mostly interested in is the expression of the generators of $K$ as words in the generators of $\mathbb{F}_l$. We begin with a disk $D$ and remove from this disk a smaller disk with the same center. Then we remove along a radius another $l-1$ holes. After a rotation of the disk by an angle of $\frac{2\pi}{p}$, we remove again along a radius another $l-1$ holes. Doing this $p$ times we construct a surface $Y$ with $(l-1)p+1$ holes. Since, by construction, this surface is invariant by a rotation $r$ of order $p$ we may quotient out $Y$ by $\langle r\rangle$ to construct a new surface $X$ with exactly $l$ holes. See  Figure \ref{pict1}.

Define $\psi:Y\rightarrow X$ the projection identifying points in $Y$ modulo $\langle r\rangle$. This is a Galois cover of $X$ of order $p$. Therefore, $K:=\pi_1(Y)$ is a normal subgroup of index $p$ in $\pi_1(X)$. Since $X$ is a disk with $l$ holes, the group $\pi_1(X)$ is freely generated by  loops around each hole, we will denote them $x_1,\dots, x_{l}$.

On the other hand, a system of free generators for $K$ is given by loops around the $(l-1)p+1$ holes which can be identified to $x_1^p$ and $x_1^{i}x_jx_1^{-i}$ with $0\leq i\leq p-1$ and $2\leq j\leq l$ (see  Figure \ref{pict1}).    

\begin{figure}
\centering

\begin{tikzpicture}[scale=0.9]

\draw[ultra thick] (0,7) ellipse (6 and 2.7);
\draw[ultra thick] (0,7) ellipse (1.2 and 0.5);
\draw[ultra thick](2,7) ellipse (0.3 and 0.175);
\draw[ultra thick](5.5,7) ellipse (0.3 and 0.175);
\draw[ultra thick](1.5,7.75) ellipse (0.3 and 0.175);
\draw[ultra thick](3.5,8.75) ellipse (0.3 and 0.175);
\path [fill=gray, even odd rule, fill opacity = 0.1] (0,7) ellipse(6 and 2.7) (0,7) ellipse (1.2 and 0.5) (2,7) ellipse(0.3 and 0.175) (5.5,7) ellipse(0.3 and 0.175) (1.5,7.75) ellipse(0.3 and 0.175) (3.5,8.75) ellipse (0.3 and 0.175);;

\node at (0,7.5){$\blacktriangleleft$};
\node at (2,7.175){$\blacktriangleleft$};
\node at (5.5,7.175){$\blacktriangleleft$};
\node at (1.5,7.925){$\blacktriangleleft$};
\node at (3.5,8.925){$\blacktriangleleft$};
\draw[thick,loosely dotted](2.4,7)--(5.1,7);
\draw[thick,loosely dotted](1.7,7.85)--(3.3,8.65);

\node at (3.75,6.6) [xscale=1,yscale=1,rotate=0] {$\underbrace{\quad\quad\quad\quad\quad\quad\quad\quad\quad}$};
\node at (3.75,6.2){$l-1$ holes};

\node at (2.1,8.55) [xscale=1,yscale=1,rotate=-152] {$\underbrace{\quad\quad\quad\quad\quad\quad\quad}$};
\node at (2,8.85)[rotate=28]{$l-1$ holes};

\node at (-1.2,7.45) {$x_1^p$};
\node at (2.45,7.2) {$x_2$};
\node at (5.1,7.2) {$x_l$};
\node at (2.6,7.75) {$x_1x_2x_1^{-1}$};
\node at (3.7,8.4) {$x_1x_lx_1^{-1}$};

\draw[thick,->] (6.15,7).. controls (6,7.8) and (5.8,8) .. (5.28,8.5) .. controls (4.5,9) .. (4.2,9.1);
\node[text width=2cm] at (4.2,9.8) { rotation of angle $\frac{2\pi}{p}$};

\draw[ultra thick,dotted,->](1,7.9)..controls (0,8.2) and (-2.5,7.8) .. (-2.5,7)..controls (-2.5,6.2) and (0,5.8) .. (1.5,6.2);
\node[text width=1.1cm] at (-3,7){$p$ lines of holes};

\node at (-4.2,9.8)[xscale=2,yscale=2,rotate=0]{Surface $Y$};

\draw[ultra thick] (0,7) --(0,6.5);
\draw[ultra thick,dash phase=1mm,dash pattern=on 1mm off 1mm] (0,6.5) -- (0,4.3);
\draw[ultra thick,->>] (0,4.3)--(0,-1);

\node at (1.2,3){projection $\psi$};

\node at (-4.2,1.8)[xscale=2,yscale=2,rotate=0]{Surface $X$};

\draw[ultra thick] (0,-1) ellipse (6 and 2.7);
\draw[ultra thick] (0,-1) ellipse (1.2 and 0.5);
\draw[ultra thick](2,-1) ellipse (0.3 and 0.175);
\draw[ultra thick](5.5,-1) ellipse (0.3 and 0.175);
\path [fill=gray, even odd rule, fill opacity = 0.1] (0,-1) ellipse(6 and 2.7) (0,-1) ellipse (1.2 and 0.5) (2,-1) ellipse(0.3 and 0.175) (5.5,-1) ellipse(0.3 and 0.175);;

\node at (0,-1.5){$\blacktriangleright$};
\node at (2,-0.825){$\blacktriangleleft$};
\node at (5.5,-0.825){$\blacktriangleleft$};
\draw[thick,loosely dotted](2.4,-1)--(5.1,-1);

\node at (3.75,-1.4) [xscale=1,yscale=1,rotate=0] {$\underbrace{\quad\quad\quad\quad\quad\quad\quad\quad\quad}$};
\node at (3.75,-1.8){$l-1$ holes};

\node at (-1.2,-1.45) {$x_1$};
\node at (2.45,-0.8) {$x_2$};
\node at (5.1,-0.8) {$x_l$};

\end{tikzpicture}
\caption{Construction of a normal subgroup of index $p$ in a free group over $l$ generators} \label{pict1}
\end{figure}
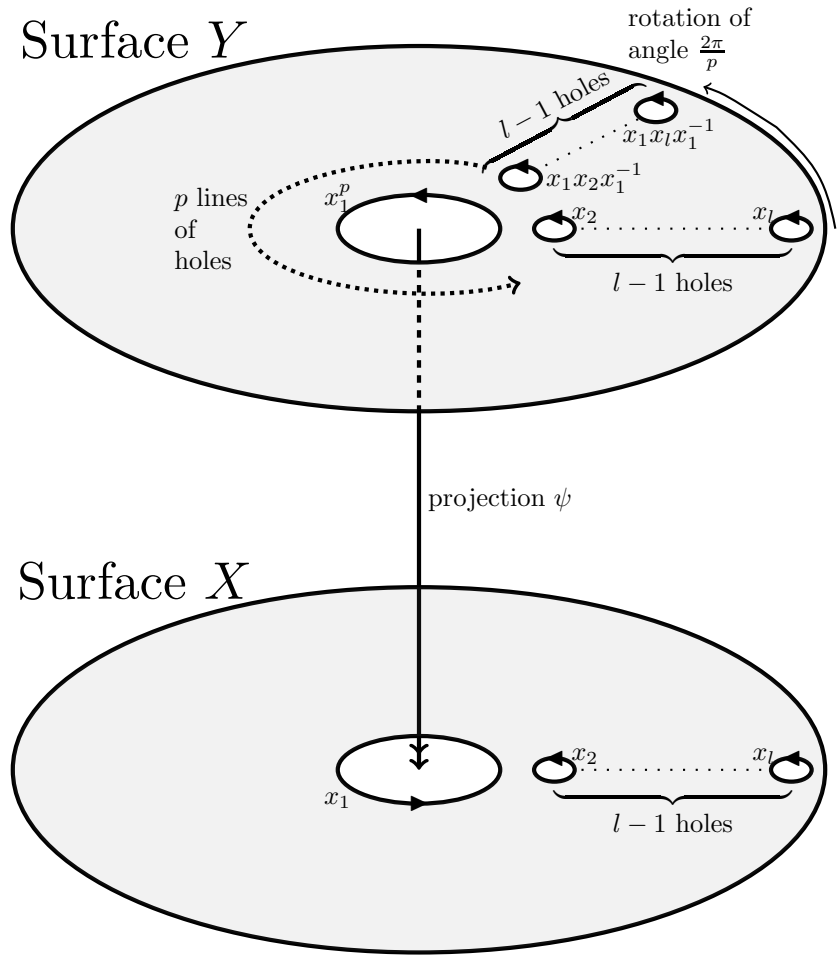

\end{remark}

Since $\langle\overline{M_c}\rangle$ acts by conjugation on $\overline{D}$, it also acts on $\overline{D}^k$ by simultaneously acting on each coordinates. The next proposition relies on the results of Section \ref{diagbyfin}, especially the explicit computations of cohomology groups.

\begin{proposition}\label{Hrhofreegroup}

Let $l\geq 2$, $p$ be a prime number. Then any pseudo-component in $\chi^i_{Sing}(\mathbb{F}_l,PSL(p,\mathbb{C}))$ is homeomorphic to $\left(\overline{D}^{l-1}-\{(\overline{I_p},\dots,\overline{I_p})\}\right)/ \langle \overline{M_c}\rangle$.

\end{proposition}

\begin{proof}

Using Proposition \ref{onesinglphiKfree}, we only construct the homeomorphism for a unique pseudo-component.  Let $x_1,\dots,x_l$ be  a system of free generators for $\mathbb{F}_l$. Then Remark \ref{freegrponeK} justifies that the subgroup $K_0$ generated by $X:=x_1^p$ and $Y_{i,j}:=x_1^{i}x_jx_1^{-i}$  with $0\leq i\leq p-1$ and $2\leq j\leq l$ is freely generated by these generators and is a normal subgroup of index $p$ in $\mathbb{F}_l$.

According to Theorem \ref{homeobarHrho}, $\varphi(\overline{\mathcal{H}_{\overline{\rho}_{K_0}}}^i)$ is homeomorphic to $\overline{\mathcal{H}_{\overline{\rho}_{K_0}}}^i$.  By Proposition \ref{barHrho}, we need to compute $H^1(\mathbb{F}_l,\overline{D}_{\overline{\rho}_{K_0}})$.  Following the notation of Proposition \ref{Transgr}, we choose $\gamma_0:=x_1$. This proposition states that there is an homeomorphism between

$$H^1(\mathbb{F}_l,\overline{D}_{\overline{\rho}_{K_0}})\text{ and } \left\lbrace f\in \Hom(K_0,\overline{D})\left|\begin{array}{l}f(x_1^p)=\overline{I_p}\\f(x_1\gamma x_1^{-1})=\overline{M_c}\cdot f(\gamma)\text{, } \forall \gamma\in K_0\end{array}\right. \right\rbrace\text{.}$$

 Identifying  $f$ to the tuples $(f(X),(f(Y_{i,j})))$ of  images  of the free generators of $K_0$, $\Hom(K_0,\overline{D})$ is equal to $\overline{D}^{1+p(l-1)}$.  Furthermore, $f(x_1^p)=\overline{I_p}$ means that $f(X)=\overline{I_p}$. Whereas  $f(x_1\gamma x_1^{-1})=\overline{M_c}\cdot f(\gamma)$, for all  $\gamma\in K_0$ is equivalent to $f(Y_{i+1,j})=\overline{M_c}\cdot f(Y_{i,j})$ for $j=2,\dots,l$ and $i=0,\dots,p-2$.

Therefore, the map sending $f$ to $(f(Y_{0,2}),\dots,f(Y_{0,l}))$ defines a natural identification between the space above and $\overline{D}^{l-1}$.  As a result, there is an homeomorphism between $H^1(\Gamma,\overline{D}_{\overline{\rho}_{K_0}})$ and $\overline{D}^{l-1}$. This leads to an homeomorphism between $H^1(\Gamma,\overline{D}_{\overline{\rho}_{K_0}})/\langle \overline{M_c}\rangle$ and $\overline{D}^{l-1}/\langle \overline{M_c}\rangle$. By Proposition \ref{barHrho}, $\overline{\mathcal{H}_{\overline{\rho}_{K_0}}}$ is homeomorphic to $\overline{D}^{l-1}/\langle \overline{M_c}\rangle$.  Finally, in Remark \ref{Hirred}, we have seen that $\overline{\mathcal{H}_{\overline{\rho}_{K_0}}}$ has only one point which is a conjugacy class of  non-irreducible representations. It is associated to the zero element in $H^1(\Gamma,\overline{D}_{\overline{\rho}_{K_0}})/\langle \overline{M_c}\rangle$. As a result,  $\varphi(\overline{\mathcal{H}_{\overline{\rho}_{K_0}}}^i)$ is homeomorphic to $\left(\overline{D}^{l-1}-\{(\overline{I_p},\dots,\overline{I_p})\}\right)/ \langle \overline{M_c}\rangle$. \end{proof}

\begin{corollary}\label{orblocfree}

Let $l\geq 2$ and $p$ be a prime number. Then the singular locus of $\chi^i(\mathbb{F}_l,PSL(p,\mathbb{C}))$ is connected of dimension $(p-1)(l-1)$. 

\end{corollary}
\begin{proof}

Equation \ref{decompbadlocus2} implies that the singular locus  $\chi^i_{Sing}(\mathbb{F}_l,PSL(p,\mathbb{C}))$ is a union of pseudo-components. By Theorem \ref{combi}, any two pseudo-components intersect. Since they are all connected by Proposition \ref{Hrhofreegroup}, the union itself is connected.

Proposition \ref{Hrhofreegroup} also implies that the   dimension of a pseudo-component is $(p-1)(l-1)$. Since the intersection between any two pseudo-components is a finite number of point and we have a finite number of these sets, the topological dimension of $\chi^i_{Sing}(\mathbb{F}_l,PSL(p,\mathbb{C}))$ is also $(p-1)(l-1)$.   \end{proof}

\begin{remark}\label{irredcompfree}

Since $\left(\overline{D}^{l-1}-\{(\overline{I_p},\dots,\overline{I_p})\}\right)/ \langle \overline{M_c}\rangle$ is  irreducible for the Zariski topology, the correspondence given in Proposition \ref{Hrhofreegroup} implies that the pseudo-components  are the irreducible components of $\chi^i_{Sing}(\mathbb{F}_l,PSL(p,\mathbb{C}))$.\end{remark}

\section{The closed surface group case}\label{surfgrpcase}

Let   $\Sigma_g$ be a closed  surface of genus $g\geq 2$. We are studying the singular locus of the character variety of the fundamental group of $\pi_1(\Sigma_g)$ in $PSL(p,\mathbb{C})$. Basically, our study will follow the lines of the free group case. From the Dehn-Nielsen-Baer theorem (see \cite{F-M} theorem 8.1), we know that the orientation preserving mapping class group of a surface is a subgroup of index $2$ in $\Out(\pi_1(\Sigma_g))$. We denote  $\Out^+(\pi_1(\Sigma_g))$ this subgroup of index $2$ and $\Aut^+(\pi_1(\Sigma_g))$ the corresponding subgroup of automorphisms. 

The set of normal subgroups $K$ of index $p$   in $\pi_1(\Sigma_g)$ is in bijective correspondence with the set of subgroups of index $p$ in $\pi_1(\Sigma_g)^{Ab}=\mathbb{Z}^{2g}$. Furthermore, it is known that $Sp(2g,\mathbb{Z})$ acts transitively on these subgroups. Since the canonical map from $\Out^+(\pi_1(\Sigma_g))$ on $Sp(2g,\mathbb{Z})$ is known to be surjective (c.f. \cite{F-M}, theorem 6.4), it follows that $\Aut^+(\pi_1(\Sigma_g))$ acts transitively on the set of normal subgroups of index $p$ in $\pi_1(\Sigma_g)$. Using this, we have :

\begin{proposition}\label{onesinglphiKsurfgrp}

Let $g\geq 2$ and $K,K'$ be two normal subgroups of index $p$ in $\pi_1(\Sigma_g)$. Then, there exists an automorphism $\phi$ in $\Aut^+(\pi_1(\Sigma_g))$ such that the precomposition by $\phi$ induces an homeomorphism between $\varphi(\overline{\mathcal{H}_{\overline{\rho}_K}}^i)$ and $\varphi(\overline{\mathcal{H}_{\overline{\rho}_{K'}}}^i)$.
\end{proposition}

\begin{proof}
The same proof as for Proposition \ref{onesinglphiKfree}.   \end{proof}

\begin{remark}\label{surfgrponeK}
This proposition implies that we only need to focus on one single pseudo-component. Let $a_1,b_1$,\dots,$a_g,b_g$ be the standard generators of $\pi_1(\Sigma_g)$ :
$$\pi_1(\Sigma_g)=\langle a_1,b_1,\dots,a_g,b_g\mid \prod_{i=1}^g[a_i,b_i]=1\rangle\text{.}$$
Then define $K$ to be the group generated $a_1^p$, $b_1$, $a_1^{i}a_ja_1^{-i}$, $a_1^{i}b_ja_1^{-i}$ for $0\leq i\leq p-1$ and $2\leq j\leq g$. Replacing the disks in Figure \ref{pict1} by tori,  one may identify $K$ to the fundamental group of a $p$-fold Galois cover of $\Sigma_g$. As a result, $K$ is a normal subgroup of index $p$ in  $\pi_1(\Sigma_g)$. Furthermore $K$ is the fundamental group of a closed topological surface of genus $1+(g-1)p$ and $\pi_1(K)^{Ab}$ is isomorphic to $\mathbb{Z}^{2+2(g-1)p}$ freely generated by the images of $a_1^p$, $b_1$, $a_1^{i}a_ja_1^{-i}$, $a_1^{i}b_ja_1^{-i}$ for $0\leq i\leq p-1$ and $2\leq j\leq g$. 

\end{remark}

\begin{proposition}\label{Hrhosurfgroup}

Let $g\geq 2$, $\Sigma_g$ be a closed surface of genus $g$ and  $p$ be a prime number. Then any pseudo-component in $\chi^i_{Sing}(\pi_1(\Sigma_g),PSL(p,\mathbb{C}))$ is homeomorphic to  $\left(\langle \overline{D(\xi)}\rangle\times \overline{D}^{2(g-1)}-\{(\overline{I_p},\dots,\overline{I_p})\}\right)/ \langle \overline{M_c}\rangle$.

\end{proposition}

\begin{proof}

Using Proposition \ref{onesinglphiKsurfgrp}, we only construct the homeomorphism for a unique pseudo-component. Let $a_1,b_1$,\dots,$a_g,b_g$ be the standard generators of $\pi_1(\Sigma_g)$. Then, Remark \ref{surfgrponeK} justifies that the subgroup $K_0$ generated by $A:=a_1^p$, $B:=b_1$, $A_{i,j}:=a_1^{i}a_ja_1^{-i}$ and $B_{i,j}:=a_1^{i}b_ja_1^{-i}$    with $0\leq i\leq p-1$ and $2\leq j\leq g$ is  a normal subgroup of index $p$ in $\pi_1(\Sigma_g)$. Furthermore,  its abelianization is freely generated by the images of $A$, $B$, $A_{i,j}$ and $B_{i,j}$.

According to Theorem \ref{homeobarHrho}, $\varphi(\overline{\mathcal{H}_{\overline{\rho}_{K_0}}}^i)$ is homeomorphic to $\overline{\mathcal{H}_{\overline{\rho}_{K_0}}}^i$.  By Proposition \ref{barHrho}, we need to compute $H^1(\pi_1(\Sigma_g),\overline{D}_{\overline{\rho}_{K_0}})$.  Following the notation of Proposition \ref{Transgr}, we choose $\gamma_0:=a_1$. This proposition states that there is an homeomorphism between $H^1(\pi_1(\Sigma_g),\overline{D}_{\overline{\rho}_{K_0}})\text{ and }$
$$ \left\lbrace f\in \Hom(K_0,\overline{D})\left|\begin{array}{l}f(a_1^p)=\overline{I_p}\\f(a_1\gamma a_1^{-1})=\overline{M_c}\cdot f(\gamma)\text{, } \forall \gamma\in K_0\end{array}\right. \right\rbrace\text{.}$$

  Identifying  $f$ to the tuples of images $(f(X),f(Y),(f(X_{i,j})),f(Y_{i,j}))$ of the free generators of $K_0^{Ab}$, $\Hom(K_0,\overline{D})$ is equal to $\overline{D}^{1+p(g-1)}$.   $f(a_1^p)=\overline{I_p}$ means that $f(A)=\overline{I_p}$. Remark that 
$$a_1^{-1}Ba_1=a_1^{-1}b_1a_1=b_1(a_1b_1)^{-1}[b_1,a_1](a_1b_1)=b_1(a_1b_1)^{-1}\prod_{j=2}^g[a_i,b_i]a_1b_1\text{.}$$
Therefore  $f(a_1Ba_1^{-1})=f(B)$. Whence $f(a_1\gamma a_1^{-1})=\overline{M_c}\cdot f(\gamma)$, for all  $\gamma\in K_0$ is equivalent to $f(B)=f(a_1Ba_1^{-1})=\overline{M_c}\cdot f(B)$, $f(A_{i+1,j})=\overline{M_c}\cdot f(A_{i,j})$ and $f(B_{i+1,j})=\overline{M_c}\cdot f(B_{i,j})$ for $j=2,\dots,l$ and $i=0,\dots,p-2$.

Thus, the map sending $f$ to $(f(B),f(A_{0,2}),f(B_{0,2}),\dots,f(A_{0,g}),f(B_{0,g}))$ defines a natural identification between the space above and $\langle \overline{D(\xi)}\rangle\times\overline{D}^{2(g-1)}$.  As a result, there is an homeomorphism between $H^1(\Gamma,\overline{D}_{\overline{\rho}_{K_0}})$ and $\langle \overline{D(\xi)}\rangle\times\overline{D}^{2(g-1)}$.

The end of the proof is similar to the free group case (Proposition \ref{Hrhofreegroup}). \end{proof}

Unlike the free group case,     neither the pseudo-components nor the singular locus are connected. To have a more refined statement, we introduce an invariant. Let $g\geq 2$ and $\Sigma_g$ be a closed   surface of genus $g$  whose fundamental group is generated by  $a_1,b_1,\dots,a_g,b_g$ verifying one single relation $[a_1,b_1]\cdots[a_g,b_g]=1$.

If $\rho$ is a representation from $\pi_1(\Sigma_g)$ to $PSL(p,\mathbb{C})$, we arbitrarily choose  for each $\gamma\in \pi_1(\Sigma_g)$, $\hat{\rho}(\gamma)\in SL(p,\mathbb{C})$ such that $\overline{\hat{\rho}(\gamma)}=\rho(\gamma)$. Since $\rho$ is a representation in $PSL(p,\mathbb{C})$ we have that :
$$e(\rho):=\prod_{i=1}^g[\hat{\rho}(a_i),\hat{\rho}(b_i)]\in Z(SL(p,\mathbb{C}))\text{.} $$
It is a straightforward verification that $e(\rho)$ only depends on $\rho$ (and not on the chosen lifts). It is called the \textit{Euler invariant}.  It is  invariant by conjugation. We shall see, in the next lemma, that it is also invariant by the action of $\Aut^+(\pi_1(\Sigma_g))$. 

\begin{lemma}\label{mcginv}

Let $g\geq 2$, $\Sigma_g$ be a closed surface of genus $g\geq 2$, $\phi\in \Aut^+(\pi_1(\Sigma_g))$ and $\rho\in \Hom(\pi_1(\Sigma_g),PSL(p,\mathbb{C}))$. Then $e(\rho\circ \phi)=e(\rho)$.

\end{lemma}

\begin{proof}

Let $\phi$ be in $\Aut^+(\pi_1(\Sigma_g))$ and $\rho\in \Hom(\pi_1(\Sigma_g),PSL(p,\mathbb{C}))$. Consider the central exact sequence defining $PSL(p,\mathbb{C})$ :

$$\xymatrix{1\ar[r]&Z(SL(p,\mathbb{C}))\ar[r]&SL(p,\mathbb{C})\ar[r]&PSL(p,\mathbb{C})\ar[r]&1}$$

Let $[z]\in H^2(PSL(p,\mathbb{C}),Z(SL(p,\mathbb{C})))$ representing this exact sequence.   

Then, $[\rho^*z]$ in  $H^2(\pi_1(\Sigma_g),Z(SL(p,\mathbb{C})))$ is invariant by   $Aut^+(\pi_1(\Sigma_g))$ since it leaves invariant the volume form generating $H^2(\Gamma,\mathbb{Z})$. Using   Poincar\'e duality, $H^2(\pi_1(\Sigma_g),Z(SL(p,\mathbb{C})))$ and  $H^0(\pi_1(\Sigma_g),Z(SL(p,\mathbb{C}))^*)=Z(SL(p,\mathbb{C}))^*$ are in  natural dual pairing. It gives an isomorphism between $H^2(\pi_1(\Sigma_g),Z(SL(p,\mathbb{C})))$ and $Z(SL(p,\mathbb{C}))$. One can check that through this isomorphism, $[\rho^*z]$ will be sent to $e(\rho)$. Whence the result. \end{proof}

It leads to the   following topological result.

\begin{corollary}\label{topsurfgrpHrhoK}

Let $g\geq 2$, $\Sigma_g$ be a closed surface of genus $g$, $p$ be a prime number and $K$ be a normal subgroup of index $p$ in $\pi_1(\Sigma_g)$. Then $\varphi(\overline{\mathcal{H}_{\overline{\rho}_{K}}}^i)$ has exactly $p$ connected components which are the fibers over the Euler invariant.  Furthermore

\begin{itemize}
\item $e^{-1}(I_p)\cap \varphi(\overline{\mathcal{H}_{\overline{\rho}_{K}}}^i)$ is homeomorphic to $(\overline{D}^{2(g-1)}-\{(1,\dots, 1)\})/\langle \overline{M_c}\rangle$.

\item for $1\leq k\leq p-1$, $e^{-1}(\xi^kI_p)\cap \varphi(\overline{\mathcal{H}_{\overline{\rho}_{K}}}^i)$ is homeomorphic to $\overline{D}^{2(g-1)}/\langle \overline{M_c}\rangle$.
\end{itemize}

\end{corollary}

\begin{proof}

Proposition \ref{onesinglphiKsurfgrp} which states that $\Aut^+(\pi_1(\Sigma_g))$ acts transitively on components and Lemma \ref{mcginv} which states that the Euler invariant is invariant by action of $\Aut^+(\pi_1(\Sigma_g))$ imply that we only need to do it for one normal subgroup of index $p$ in $\pi_1(\Sigma_g)$. Define $a_1,b_1,\dots, a_g,b_g$ to be its standard generators and, like before define $K_0$ to be the subgroup generated by $a_1^p,b_1$, $a_1^{i}a_ja_1^{-i}$, $a_1^{i}b_ja_1^{-i}$ for $1\leq i\leq p$ and $2\leq j\leq g$.

Define a group morphism  $\rho_k:\pi_1(\Sigma_g)\rightarrow \overline{D}\rtimes\langle \overline{M_c}\rangle$ by $\rho_k(a_1):=\overline{M_c} $, $\rho_k(b_1):=\overline{D(\xi)}^k$ and $\rho_k(a_i)=\rho_k(b_i)= \overline{D(\xi)}$ for $2\leq i\leq g$. Since  $\rho_k(a_1)$ and $\rho_k(a_2)$ generate an irreducible subgroup of $PSL(p,\mathbb{C})$,   $\rho_k$ is irreducible. Finally,  $\Ker(q\circ\rho_k)$   contains $K_0$ and is therefore equal to it because they have the same index in $\pi_1(\Sigma_g)$. Therefore, $[\rho_k]$ belongs to $\varphi(\overline{\mathcal{H}_{\overline{\rho}_{K_0}}}^i)$.

Taking the following lifts $\hat{\rho_k}(a_1):=M_c$, $\hat{\rho_k}(b_1):=D(\xi)^k$, $\hat{\rho_k}(a_i):=D(\xi)$ and $\hat{\rho_k}(b_i):=D(\xi)$ for $i\geq 2$, the Euler invariant of $\rho_k$ is  $e(\rho_k)=\xi^{-k}I_p$. As a result, the map $e:\varphi(\overline{\mathcal{H}_{\overline{\rho}}}^i)\rightarrow Z(SL(p,\mathbb{C}))$ is onto. Whence $\varphi(\overline{\mathcal{H}_{\overline{\rho}_{K_0}}}^i)$ has at least $p$ connected components (the $p$ fibers over $e$). Proposition \ref{Hrhosurfgroup} implies that it has exactly $p$ connected components. Therefore, the $p$ connected components of $\varphi(\overline{\mathcal{H}_{\overline{\rho}_{K_0}}}^i)$ are its fibers over $e$.

\bigskip

Remark that, through the homeomorphism of Proposition \ref{Hrhosurfgroup}, which sends $\varphi(\overline{\mathcal{H}_{\overline{\rho}_{K_0}}}^i)$  to $\left(\langle \overline{D(\xi)}\rangle\times \overline{D}^{2(g-1)}-\{(\overline{I_p},\dots,\overline{I_p})\}\right)/ \langle \overline{M_c}\rangle$, the conjugacy class of $\rho_k$ is sent to $(\overline{D(\xi)}^k,\overline{D(\xi)},\dots,\overline{D(\xi)})$.

As a result, the connected component $e^{-1}(I_p)$ of the conjugacy class of $\rho_0$ is  homeomorphic to $\left( \overline{D}^{2(g-1)}-\{(\overline{I_p},\dots,\overline{I_p})\}\right)/ \langle \overline{M_c}\rangle$. For $1\leq k\leq p-1$, the connected component $e^{-1}(\xi^{-k}I_p)$ of the conjugacy class of $\rho_k$ is  homeomorphic to $\overline{D}^{2(g-1)}/ \langle \overline{M_c}\rangle$. \end{proof}

\begin{remark}\label{numbconcomp}
Applying this corollary,  $\chi_{Sing}^i(\pi_1(\Sigma_g),PSL(p,\mathbb{C}))$ has at least $p$ connected components given by their Euler invariants. Our next order of business is to show that each fiber above the Euler invariant is connected. 

Basically, the idea would be to apply Theorem \ref{combi} (like in the free group case). However, it might happen that $\varphi(\overline{H_{\overline{\rho}_K}}^i)\cap\varphi(\overline{H_{\overline{\rho}_{K'}}}^i)\cap e^{-1}(\xi^k I_p)$ is empty. The necessary and sufficient conditions for this set to be  not empty are given in Proposition \ref{interempty}. To do so, it is necessary to consider $\pi_1(\Sigma_g)^{Ab}/p\pi_1(\Sigma_g)^{Ab}$ as a symplectic space.

\end{remark}

By  \textit{symplectic (vector) space}, we mean a vector space endowed with a symplectic bilinear form. Let $(E_0,\omega)$ be a  symplectic space with  a linear basis $(x_1,y_1,\dots,x_g,y_g)$. It is a symplectic basis   if $\omega(x_i,x_j)=0=\omega(y_i,y_j)$ for all $i,j$ and $\omega(x_i,y_j)=1$ if $i=j$ and $0$ else. We will use the fact that any vector space endowed with a symplectic bilinear form admits a symplectic basis and any linear endomorphism sending a symplectic basis to a symplectic basis is a symplectic transformation. We denote $Sp(E_0,\omega)$ the group of symplectic transformations for $(E_0,\omega)$. 

Furthermore, if $E$ is a subspace of codimension $2$ in $E_0$ then  either $\omega_{|E}$ is non-degenerate or it has a kernel of dimension $2$ (because $\omega$ is non-degenerate). In the first case, we say that $E$ is \textit{non-degenerate}, in the second case we say that $E$ is \textit{degenerate}.

\begin{lemma}\label{twoorbits}

Let $g\geq 1$ and $(E_0,\omega)$ be a symplectic $K$-linear space of dimension $2g$. Then, there are, two orbits in $\left\lbrace (E,E')\left|\begin{array}{c}E,E'\text{ are hyperplanes in $E_0$}\\ E\neq E'\end{array}\right.\right\rbrace$ for the action of $Sp(E_0,\omega)$ depending whether $E\cap E'$ is degenerate or non-degenerate.

\end{lemma}

\begin{proof}

Let $(x_1,y_1,\dots,x_g,y_g)$ be a symplectic basis for $E_0$. Let $E,E'$ be two hyperplanes in $E_0$ verifying $E\neq E'$. 

If $E\cap E'$ is non-degenerate we shall find a symplectic transformation sending $E$ to $\Span (x_1,x_2,y_2,\dots,x_g,y_g)$ and $E'$ to $\Span (y_1,x_2,y_2,\dots,x_g,y_g)$.

If $E\cap E'$ is degenerate we shall find a symplectic transformation sending $E$ to $\Span (y_1,x_2,y_2,x_3,y_3\dots,x_g,y_g)$ and $E'$ to $\Span (x_1,y_1,y_2,x_3,y_3\dots,x_g,y_g)$.

\bigskip

Assume $F:=E\cap E'$ is non-degenerate. Let $(x_2',y_2',\dots, x_g',y_g')$ be a symplectic basis of $F$. Let $u$ be in $E$ such that $E=\Span(u)\oplus F$. Since $F$ is non-degenerate, there exists $f_0\in F$ such that $\omega(u,\cdot)_{|F}=\omega(f_0,\cdot)_{|F}$. Therefore $v:=u-f_0$ is orthogonal to $F$ and $E=\Span(v)\oplus^{\perp} F$. Likewise, there is $v'\in E'$ verifying $E'=\Span(v')\oplus^{\perp} F$. Since  $E_0=\Span(v,v')\oplus^{\perp} F$ and  $E_0$ and $F$ are both non-degenerate, $\Span(v,v')$ is non degenerate, whence $\omega(v,v')\neq 0$.

 Let $v'':=\omega(v,v')^{-1} v'$ then $(v,v'',x_2',y_2',\dots,x_g',y_g')$ is a symplectic basis of $E_0$. Therefore, the linear map $\psi$ sending this symplectic basis to the initial one $(x_1,y_1,\dots,x_g,y_g)$ is the wanted symplectic transformation.

\bigskip

Assume $F:=E\cap E'$ is degenerate. Let $K:=F^{\perp}$. Then $F=K\oplus^{\perp} F_0$ where $F_0$ is non-degenerate. Let $(x_3',y_3',\dots, x_g',y_g')$ be a symplectic basis of $F_0$.

Let $u$ be in $E$ such that $E=\Span(u)\oplus F$. Since $F_0$ is non-degenerate, there exists $f_0\in F_0$ such that $\omega(u,\cdot)_{|F_0}=\omega(f_0,\cdot)_{|F_0}$. Therefore $v:=u-f_0$ is orthogonal to $F_0$ and $E=(\Span(v)\oplus K)\oplus^{\perp} F_0$. Likewise, there is $v'\in E'$ verifying $E'=(\Span(v')\oplus K)\oplus^{\perp} F_0$. We remark that $E_0=(\Span(v,v')\oplus K)\oplus^{\perp} F_0$ and therefore $\Span(v,v')\oplus K$ is non-degenerate.

Since $\Span(v,v')\cap F$ is trivial and $F=K^{\perp}$, there is $y_1'\in K$ verifying $\omega(v,y_1')=1$ and $\omega(v',y_1')=0$. If  $v'':=v'-\omega(v,v')y_1'$, then  $v''$ is both orthogonal to $v'$ and $y_1'$. Furthermore $E''=(\Span(v'')\oplus K)\oplus^{\perp} F_0$. Finally, there is $y_2'\in K$ verifying $\omega(v'',y_2')=1$ and $\omega(v,y_2')=0$. Therefore $(v,y_1',v'',y_2',x_3',y_3',\dots, x_g',y_g')$ is a symplectic basis and the linear map  sending this symplectic basis to $(x_1,y_1,\dots,x_g,y_g)$ is the wanted symplectic transformation. \end{proof}

 \begin{remark}\label{irredcompsurf} Let $g\geq 2$. Then   $\pi_1(\Sigma_g)^{Ab}/p\pi_1(\Sigma_g)^{Ab}$ is isomorphic to $(\mathbb{Z}/p)^{2g}$. We denote $\Phi$ the natural projection of $\pi_1(\Sigma_g)$ onto $E_{p,g}:=(\mathbb{Z}/p)^{2g}$. Remark that there is a natural symplectic form $\omega$ on $E_{p,g}$ which is invariant by the action of $\Aut^+(\pi_1(\Sigma_g))$. If $(a_1,b_1,\dots,a_g,b_g)$ is a standard system of generator for $\pi_1(\Sigma_g)$ then $(\Phi(a_1),\Phi(b_1),\dots,\Phi(a_g),\Phi(b_g))$ is naturally a symplectic basis of $E_{p,g}$ for $\omega$. \end{remark}

\begin{proposition}\label{interempty}

Let $g\geq 2$, $p$ be a prime number, $\Sigma_g$ be a closed surface of genus $g$, $K,K'$ be two different normal subgroups of index $p$ in $\pi_1(\Sigma_g)$ and $k$ be an integer between $0$ and $p-1$. Then, the cardinal of $e^{-1}(\xi^kI_p)\cap\varphi(\overline{\mathcal{H}_{\overline{\rho}_K}}^i)\cap\varphi(\overline{\mathcal{H}_{\overline{\rho}_{K'}}}^i)$  is :

$$\left\lbrace\begin{array}{cl}1&\text{ if $k\neq 0$ and $\Phi(K\cap K')$\text{ is non-degenerate}}\\ 0&\text{ if $k=0$ and $\Phi(K\cap K')$\text{ is non-degenerate}}\\ 0&\text{ if $k\neq 0$ and $\Phi(K\cap K')$\text{ is degenerate}}\\ p-1&\text{ if $k=0$ and $\Phi(K\cap K')$\text{ is degenerate}}\end{array}\right.$$

\end{proposition}

\begin{proof}

Lemma \ref{twoorbits} justifies that we only need to show this property for one example when $\Phi(K\cap K')$ is non-degenerate and one example when $\Phi(K\cap K')$ is degenerate. We denote $a_1,b_1,\dots,a_g,b_g$ a system of standard generators for $\pi_1(\Sigma_g)$.

\bigskip

The non-degenerate case. We define $K:=\Phi^{-1}(E)$ and $K':=\Phi^{-1}(E')$ where 

 $$\begin{array}{l}E:=\Span(\Phi(b_1),\Phi(a_2),\Phi(b_2),\dots,\Phi(a_g),\Phi(b_g))\\ E':=\Span(\Phi(a_1),\Phi(a_2),\Phi(b_2),\dots,\Phi(a_g),\Phi(b_g))\end{array}$$

Then $K$ and $K'$ are both normal subgroup of index $p$ in $\pi_1(\Sigma_g)$. Furthermore $\Phi(K\cap K')$ is the subspace of $E_{p,g}$ generated by $(\Phi(a_2),\Phi(b_2),\dots,\Phi(a_g),\Phi(b_g))$. Since this is a symplectic family, it follows that  $\Phi(K\cap K')$ is non-degenerate.

For $1\leq k\leq p-1$, define  a group morphism $\rho_k$ from $\pi_1(\Sigma_g)$ to $PSL(p,\mathbb{C})$ with $\rho_k(a_1):=\overline{M_c}$, $\rho_k(b_1):=\overline{D(\xi)}^k$ and $\rho(K\cap K')$ is trivial. The image of  $\rho_k$ is irreducible and its  conjugacy class belongs to $\varphi(\overline{\mathcal{H}_{\overline{\rho}_K}}^i)\cap\varphi(\overline{\mathcal{H}_{\overline{\rho}_{K'}}}^i)$.

Finally,  $e(\rho_k)=[M_c,D(\xi)^k]=\xi^{-k}I_p$. As a result,  $\rho_1$,\dots, $\rho_{p-1}$ leads to   $p-1$ different points in $\varphi(\overline{\mathcal{H}_{\overline{\rho}_K}}^i)\cap\varphi(\overline{\mathcal{H}_{\overline{\rho}_{K'}}}^i)$. Theorem \ref{combi} implies that $\varphi(\overline{\mathcal{H}_{\overline{\rho}_K}}^i)\cap\varphi(\overline{\mathcal{H}_{\overline{\rho}_{K'}}}^i)$ contains $p-1$ elements. Whence the result.

\bigskip

The degenerate case. We define $K:=\Phi^{-1}(E)$ and $K':=\Phi^{-1}(E')$ where :

$$\begin{array}{l}E:=\Span(\Phi(b_1),\Phi(a_2),\Phi(b_2),\dots,\Phi(a_g),\Phi(b_g))\\ E':=\Span(\Phi(a_1),\Phi(b_1),\Phi(b_2),\Phi(a_3),\Phi(b_3),\dots,\Phi(a_g),\Phi(b_g))\text{.}\end{array}$$

Then $K$ and $K'$ are both normal subgroup of index $p$ in $\pi_1(\Sigma_g)$. Furthermore $\Phi(K\cap K')$ is the subspace of $E_{p,g}$ generated by $\Phi(b_1)$, $\Phi(b_2)$, $\Phi(a_3)$, $\Phi(b_3)$,\dots, $\Phi(a_g)$, $\Phi(b_g)$. Since $\Phi(b_1)$ is   orthogonal to  $\Phi(K\cap K')$, $\Phi(K\cap K')$ is degenerate.

For $1\leq k\leq p-1$, define  a group morphism $\rho_k$ from $\pi_1(\Sigma_g)$ to $PSL(p,\mathbb{C})$ with $\rho_k(a_1):=\overline{M_c}$, $\rho_k(a_2):=\overline{D(\xi)}^k$ and $\rho(K\cap K')$ is trivial. The image of  $\rho_k$ is irreducible and its  conjugacy class belongs to $\varphi(\overline{\mathcal{H}_{\overline{\rho}_K}}^i)\cap\varphi(\overline{\mathcal{H}_{\overline{\rho}_{K'}}}^i)$.

Finally,   $e(\rho_k)=I_p$. Remark that if $\rho_k$ is conjugate to $\rho_{k'}$ then the bases  of $\langle\overline{D(\xi)}\rangle \times \langle\overline{M_c}\rangle$ : $(\overline{M_c},\overline{D(\xi)}^k)$ and $(\overline{M_c},\overline{D(\xi)}^{k'})$ are obtained by the action of the normalizer of $\langle\overline{D(\xi)}\rangle \times \langle\overline{M_c}\rangle$. Lemma \ref{countZp2}  implies $k=k'$. Therefore $\rho_1$,\dots, $\rho_{p-1}$ leads to  $p-1$ different points in $\varphi(\overline{\mathcal{H}_{\overline{\rho}_K}}^i)\cap\varphi(\overline{\mathcal{H}_{\overline{\rho}_{K'}}}^i)$ which all have a trivial Euler invariant. Theorem \ref{combi} implies that their conjugacy classes are the only points in the intersection, whence the result. \end{proof}

The final result relies on the following geometric lemma.

\begin{lemma}\label{sympgeom}

Let $E$ and $E'$ be two different hyperplanes of $V:=(\mathbb{Z}/p)^{2g}$ with $g\geq 2$. Let $\omega$ be a symplectic form on $V$.  We have two cases :

\begin{itemize}

\item If  $E\cap E'$ is non-degenerate, there is an hyperplane $E_0$ such that  $E\cap E_0$ and $E'\cap E_0$ are both degenerate.

\item If $E\cap E'$ is degenerate, there is an hyperplane $E_0$ such that   $E\cap E_0$ and $E'\cap E_0$ are both non-degenerate.

\end{itemize}

\end{lemma}

\begin{proof}

Let $(x_1,y_1,\dots,x_g,y_g)$ be a symplectic basis of $V$. Lemma \ref{twoorbits} justifies that it suffices to prove the lemma for one example in each case.

Let $F$ be the subspace generated by $x_2,y_2$,\dots,$x_g,y_g$, $E$ be $\Span(x_1)\oplus F$ and $E'$ be $\Span(y_1)\oplus F$.  Then $E\cap E'=F$ is non-degenerate. Let $E_0$ be the subspace generated by $x_1,y_1$, $x_2$, $x_3,y_3$,\dots,$x_g,y_g$. Then $x_2$ is a degenerate vector of $E_0\cap E$ and $E_0\cap E'$. Whence the result in this case.

Let $F_0$ be the subspace generated by $y_1,y_2,x_3,y_3$,\dots,$x_g,y_g$, $E$ be $\Span(x_1)\oplus F$ and $E'$ be $\Span(x_2)\oplus F$. Then $E\cap E'=F$ is degenerate. Let $E_0$ be the subspace generated by $x_1,x_2,y_1+y_2,x_3,y_3,\dots,x_g,y_g$. Then  $E_0\cap E$ and $E_0 \cap E'$ are both non-degenerate. Whence the result in this case. \end{proof}

Applying Corollary \ref{topsurfgrpHrhoK}, $e^{-1}(\xi^kI_p)\cap\varphi(\overline{\mathcal{H}_{\overline{\rho}_K}}^i)$ is connected of dimension $2(g-1)(p-1)$. Using Proposition \ref{interempty} and Lemma \ref{sympgeom}, this  leads to :

\begin{corollary}\label{conncomporblocsurfgrp}

Let $g\geq 2$ and $\Sigma_g$ be a closed Riemann surface of genus $g$. The orbifold locus of $\chi^i(\pi_1(\Sigma_g),PSL(p,\mathbb{C}))$ has $p$  connected components given by the fibers for the Euler invariant and its dimension is $2(g-1)(p-1)$.

\end{corollary}

\begin{remark}\label{surfgrpirredcomp}

Considering the Zariski topology, we see that for each $0\leq k\leq p-1$ and each normal subgroup $K$ of index $p$ in $\pi_1(\Sigma_g)$, $e^{-1}(\xi^kI_p)\cap\varphi(\overline{\mathcal{H}_{\overline{\rho}_K}}^i)$ is an irreducible component of $\chi_{Sing}^i(\pi_1(\Sigma_g),PSL(p,\mathbb{C}))$.

\end{remark}

\section{Orbifold singularities and algebraic singularities}\label{orbalgsing}

Like we said in the introduction, the singular locus  $\chi^i_{Sing}(\Gamma,G)$ should be understood as orbifold singularities since, when $\Gamma$ is a free group or a surface group, it is the set of orbifold singularities of a well-defined orbifold.   A priori, it does not necessarily coincide with the locus of  algebraic singularities of $\chi^i(\Gamma,G)$. In Proposition \ref{singvarcar}, we   prove  that these two notions coincide when $\Gamma$ is a free group of rank $\geq 2$ or a closed surface group of genus $g\geq 2$ and $G=PSL(p,\mathbb{C})$ with $p$ prime.

To do this, we recall some facts about   tangent spaces to representation varieties and character varieties. If $\Gamma$ is a finitely generated group , $G$ is a complex algebraic group, $\mathfrak{g}$ is its Lie algebra and $\rho:\Gamma\to G$ is a representation, then $\mathfrak{g}$ is   a $\Gamma$-module with the action given by $Ad\circ\rho$. This $\Gamma$-module will be denoted   $\mathfrak{g}_{Ad\circ\rho}$.

 In \cite{Sik}, Proposition 34 (see also \cite{L-M}, Chapter 2) Sikora proves that the Zariski tangent space to $\Homb(\Gamma,G)$ at $\rho$ (denoted $T_{\rho}\Homb(\Gamma,G)$) can be identified to this  space of $1$-cocycles $Z^1(\Gamma,\mathfrak{g}_{Ad\circ\rho})$.

 \begin{remark}\label{similar} The proof of Proposition \ref{Hrho} is  similar  to Sikora's proof. The reason is that the Zariski tangent space $T_{\rho}\Homb(\Gamma,G)$ can be identified to the fiber of $\Hom(\Gamma,G(\mathbb{C}[\epsilon]))$ above $\rho$ where $G(\mathbb{C}[\epsilon])$ is naturally isomorphic to $\mathfrak{g}\rtimes_{Ad}G$.\end{remark}

  One can show that the Zariski tangent space to the orbit $[\rho]$ at $\rho$ is exactly $B^1(\Gamma,\mathfrak{g}_{Ad\circ\rho})$. For instance, Weil (in \cite{Wei}) uses this to prove that $H^1(\Gamma,\mathfrak{g}_{Ad\circ\rho})=0$ implies the local rigidity of $\rho$ (i.e.  a neighborhood of $\rho$ in the representation variety is contained in the conjugacy class of $\rho$). In particular, any finite group representation is locally rigid.

However, in general, the Zariski tangent space at $[\rho]$  to the schematic character variety $\mathfrak{X}(\Gamma,G)$ is different from $H^1(\Gamma,\mathfrak{g}_{Ad\circ\rho})$ (e.g. in  \cite{Ben}). When $\rho$ is scheme smooth, the link between this tangent space and $H^1(\Gamma,\mathfrak{g}_{Ad\circ\rho})$  has been studied  in \cite{Sik},  Paragraph 13 by Sikora (see also Proposition 5.2 in \cite{H-P}) using Luna's \'Etale Slice Theorem. Theorem 53 in \cite{Sik} states that if $\rho$ is a scheme smooth completely reducible representation then :

\begin{equation}
\dim\left(T_0 \left(H^1(\Gamma,\mathfrak{g}_{Ad\circ\rho})//Z_G(\rho)\right)\right)=\dim\left( T_{[\rho]}X(\Gamma,G)\right)=\dim\left( T_{[\rho]}\mathfrak{X}(\Gamma,G)\right)
\label{dimesptgt}
\end{equation}

To prove Proposition \ref{singvarcar}, we need two lemmas :

\begin{lemma}\label{dimirred}

Let $\Gamma$ be either a free group or a closed surface group and $G$ be a complex reductive algebraic group.  Then, for any irreducible representation $\rho$ from  $\Gamma$ to $G$, $\dim(\chi^i(\Gamma,G))=\dim(H^1(\Gamma,\mathfrak{g}_{Ad\circ\rho}))$.

\end{lemma}

\begin{proof}
If $\Gamma$ is a free group with $l\geq 2$ generators then  $Z^1(\Gamma,\mathfrak{g}_{Ad\circ\rho}))=\mathfrak{g}^l$ and $B^1(\Gamma,\mathfrak{g}_{Ad\circ\rho}))=\mathfrak{g}/\mathfrak{g}^{Ad\circ\rho(\Gamma)}$. Since $\rho(\Gamma)$ is irreducible $\mathfrak{g}^{Ad\circ\rho(\Gamma)}=\Lie(Z(G))$. 
 $$\dim(H^1(\Gamma,\mathfrak{g}_{Ad\circ\rho}))=\dim(\mathfrak{g}^l)-\dim(\mathfrak{g})+\dim(Z(G))=(l-1)\dim(G)+\dim(Z(G)) $$ 
$$\dim(\chi^i(\Gamma,G))=\dim(\Hom(\Gamma,G))-\dim(G/Z(G))=(l-1)\dim(G)+\dim(Z(G))\text{.}$$
  
  Whence the result in the free group case. For the surface group case, it is also classical but more difficult, see \cite{Gol}.\end{proof}

The next lemma  is well-known but its proof is written for the convenience of the reader. Our interest in cyclic quotients (vector spaces quotiented by a cyclic group action) comes from Theorem \ref{classifcentr} which implies that  centralizers of non-abelian bad representations are cyclic.

\begin{lemma}\label{cyclicaction}
Let $G$ be a cyclic group of prime order $p$ generated by $g$ and assume that we have a linear action of  $G$ on $\mathbb{C}^N$ where $N\geq 0$. Then $0$ mod $G$ is an algebraic singularity of the cyclic quotient $\mathbb{C}^N//G$ if and only if  $\codim(\Fix(g))>1$. \end{lemma}

\begin{proof}
Up to some   change of basis, there are  integers $0\leq a_1\leq \cdots\leq a_N\leq p-1$ such that the action of $G$ on $\mathbb{C}^N$ is given by $g\cdot (v_1,\dots,v_N)=(\xi^{a_1}v_1,\dots,\xi^{a_N}v_N)$ where $\xi$ is a primitive $p$-th root of the unity.  To study $\mathbb{C}^N//G$, we only need to compute its coordinate ring $B$. Let $A:=\mathbb{C}[x_1,\dots,x_N]$ be the coordinate ring of $\mathbb{C}^N$. Then $B$ is  the subring of  $A$ of invariants by the action of $G$ : $B=A^G$. 

\bigskip

If $\codim(\Fix(g))=0$, then $A=B$ and $\mathbb{C}^N//G$ is  everywhere smooth. If $\codim(\Fix(g))=1$, then $B=\mathbb{C}[x_1,\dots,x_{N-1},x_N^{p}]$. In this case $B$ is isomorphic to $A$ and, again, $\mathbb{C}^N//G$ is  everywhere smooth.

\bigskip

If  $\codim(\Fix(g))>1$, then let $k<N-1$ be such that $a_k=0$ and $a_{k+1}>0$. Define $\mathfrak{m}_0:=(x_1,\dots,x_N)$ in $A$ and $\mathfrak{n}_0:=B\cap \mathfrak{m}_0$ the maximal ideal in $B$ associated to $0$ mod $G$. Then $\mathfrak{n}_0$ contains $x_1,\dots,x_k,x_{k+1}^p,\dots,x_N^p$. Let $u$ (resp. $v$) be the inverse of $a_{N-1}$ (resp $a_N$) modulo $p$,   $x_{N-1}^{u}x_N^{p-v}$ is   invariant by $G$. Therefore it belongs to $\mathfrak{n}_0$ and we have $N+1$ elements in $\mathfrak{n}_0$ which are  linearly independent in the Zariski tangent space $\mathfrak{n}_0/\mathfrak{n}_0^2$ to  $\mathbb{C}^N//G$ at $0$ mod $G$. Whence $0$ mod $G$ is an algebraic singularity. \end{proof}

The proof of the following proposition  is a generalization of the  proof of Proposition 5.8 in \cite{H-P}.

\begin{proposition}\label{singvarcar}
Let $\Gamma$ be either a free group of rank $\geq 2$ or a surface group of genus $g\geq 2$,  $p$ be a prime number and  $\rho:\Gamma\to PSL(p,\mathbb{C})$ be an irreducible representation. Then $[\rho]$ is an algebraic singularity in the schematic character variety if and only if $\rho$ is bad. 
\end{proposition}

\begin{proof} Let $\rho$ be a good representation, since $\Gamma$ is a free group or a closed surface group, $\rho$ is scheme smooth. Corollary 50 in \cite{Sik} implies that $[\rho]$ is scheme smooth   as well since $\rho$ has trivial centralizer.

\bigskip

Conversely, assume $\rho$ is a bad representation. Up to conjugation (Theorem \ref{classifcentr}) we may assume that $\rho$ is a representation into $\overline{D}\rtimes\langle\overline{M_c}\rangle$. Let $K$ be the kernel of $q\circ \rho$. For the moment, we assume that $\rho$ is not abelian and therefore its centralizer is $\langle\overline{D(\xi)}\rangle$. In order to prove that $[\rho]$ is an algebraic singularity we prove that $\dim\left( T_{[\rho]}\mathfrak{X}(\Gamma,PSL_p)\right)> \dim\mathfrak{X}(\Gamma,PSL_p)=\dim\chi(\Gamma,PSL(p,\mathbb{C}))$.

We denote $\mathfrak{sl}_p(\mathbb{C})$ the Lie algebra of $PSL(p,\mathbb{C})$. Following Equation \ref{dimesptgt}, we need to compute $H^1(\Gamma,\mathfrak{sl}_p(\mathbb{C})_{Ad\circ\rho})$. For $0\leq i,j\leq p-1$,   we denote $E_{i,j}$ the matrix with a $1$ at the $(i,j)$-th entry and $0$ everywhere else. Let $\mathfrak{d}_0$ be the Lie algebra of trace free diagonal matrices in $\mathfrak{sl}_p(\mathbb{C})$ and for $k=1,\dots,p-1$, define 
$$\mathfrak{d}_k:=\bigoplus_{i=0}^{p-1}\mathbb{C}E_{i,i+k}\text{.} $$
Since for $0\leq k\leq p-1$, $\mathfrak{d}_k$ is stable by the action of $\overline{D}$ and $\overline{M_c}$, the   decomposition of $\mathfrak{sl}_p(\mathbb{C})=\mathfrak{d}_0\oplus\dots\oplus \mathfrak{d}_{p-1}$ as $\mathbb{C}$-vector space is also a decomposition as $\Gamma$-module.  
\begin{equation}
H^1(\Gamma,\mathfrak{sl}_p(\mathbb{C})_{Ad\circ\rho})= H^1(\Gamma,\mathfrak{d}_{0,Ad\circ\rho})\oplus \bigoplus_{k=1}^{p-1}H^1(\Gamma,\mathfrak{d}_{k,Ad\circ\rho})\text{.}
\label{decompH1}
\end{equation}
For $1\leq k\leq p-1$, $0\leq i\leq p-1$,  $Ad(\overline{D(\xi)})\cdot E_{i,i+k}=\xi^{k}E_{i,i+k}$. Therefore, the action of $\overline{D(\xi)}$ on $\mathfrak{d}_k$ is the multiplication by $\xi^k$ and so is the induced action of  $\overline{D(\xi)}$ on $H^1(\Gamma,\mathfrak{d}_{k,Ad\circ\rho})$.

\bigskip

We recognize a cyclic quotient. Using Lemma \ref{cyclicaction},  we only need to prove that $\dim(H^1(\Gamma,\mathfrak{sl}_p(\mathbb{C})_{Ad\circ\rho}))-\dim (H^1(\Gamma,\mathfrak{d}_{0,Ad\circ\rho}))>1$  to end up with an algebraic singularity at the origin.

If $\Gamma=\mathbb{F}_r$, then  $\dim \left(Z^1(\mathbb{F}_r,\mathfrak{d}_{k,Ad\circ\rho})\right)=r\dim(\mathfrak{d}_k)=rp$. Since $\rho$ is irreducible, $\mathfrak{sl}_p(\mathbb{C})^{\mathbb{F}_r}=\{0\}$. Therefore $\mathfrak{d}_k^{\mathbb{F}_r}=\{0\}$ and $\dim \left(B_1(\mathbb{F}_r,\mathfrak{d}_{k,Ad\circ\rho})\right)=p-0=p$. Therefore $\dim \left(H^1(\mathbb{F}_r,\mathfrak{d}_{k,Ad\circ\rho})\right)=(r-1)p\geq 2$. 

At the end of Paragraph $6$ in \cite{Wei}, Weil gives an explicit formula for the dimension of the first cohomology groups for Fuchsian groups without parabolic elements acting on $\mathbb{C}$-vector spaces. Applying this to the surface group case,  $\dim\left(H^1(\pi_1(\Sigma_g),\mathfrak{d}_{k,Ad\circ\rho})\right)=(2g-2)p\geq 2$ for $1\leq k\leq p-1$.

In any case, Lemma \ref{cyclicaction} implies that $$\dim \left(T_0 \left(H^1(\Gamma,\mathfrak{sl}_{p}(\mathbb{C})_{Ad\circ\rho})//\langle \overline{D(\xi)}\rangle\right)\right)> \dim \left(H^1(\Gamma,\mathfrak{sl}_p(\mathbb{C})_{Ad\circ\rho})\right)\text{.}$$
 Using Lemma \ref{dimirred} and Sikora's expression for the dimension of the tangent space at $[\rho]$  (Equation \ref{dimesptgt}) which applies because $\rho$ is scheme smooth, $[\rho]$ is an algebraic singularity of the schematic character variety.

 \bigskip
 
In the free group or closed surface group case, any abelian irreducible representation is a limit of non-abelian bad representations. Since each of them is an algebraic singularity of the character variety and the set of algebraic singularities is closed, any conjugacy class of bad representations is an algebraic singularity.   \end{proof}

\begin{remark}\label{reduced}

Let $X$ be a (possibly non-reduced) scheme and $x$ be a closed point of $X$. A point $x$ is \textit{reduced} if its local ring does not contain non-trivial nilpotent elements. A representation (resp. conjugacy class of representations) is scheme smooth if and only if it  is both smooth and reduced. Using  Corollary 55 in \cite{Sik}, one sees that we can forget the "schematic" in the preceding proposition.

\end{remark}

In general, there is no  link between the set of algebraic singularities for the irreducible part of the character variety and the singular locus of the character variety. The next example is based on the idea of Sikora, c.f. \cite{Sik}, Example 42.

\begin{example}\label{algnotorb}
Let $\rho_1$ be the trivial representation of $\mathbb{Z}^2$ into $SL(2,\mathbb{C})$ and $\rho_2$  be the unique irreducible representation of the symmetric group $S_3$ into $SL(2,\mathbb{C})$.  Let $\Gamma$ be the free product of $\mathbb{Z}^2$ and $S_3$ and $\rho:=\rho_1*\rho_2$. Then $\rho$ is an irreducible representation of $\Gamma$ into $SL(2,\mathbb{C})$.

We denote $\Hom(\Gamma,SL(2,\mathbb{C}))_0$ the set of representations in $\Hom(\Gamma,SL(2,\mathbb{C}))$ such that their restriction to $S_3$ is conjugate to $\rho_2$. Since $\rho_2$ is good and locally rigid,   $\Hom(\mathbb{Z}^2,SL(2,\mathbb{C}))\times \{\rho_2\}$ is an \'etale slice for the $PSL(2,\mathbb{C})$-action by conjugation on  $\Hom(\Gamma,SL(2,\mathbb{C}))_0$. Therefore, the tangent space to the character variety at $[\rho]$ is simply the tangent space at $\rho_1$ to $\Hom(\mathbb{Z}^2,SL(2,\mathbb{C}))$.

Example 42 in \cite{Sik} justifies that $\rho_1$ is an algebraic singularity of the representation variety $\Hom(\mathbb{Z}^2,SL(2,\mathbb{C}))$. Therefore, $[\rho]$ is an algebraic singularity of $\chi^i(\Gamma,SL(2,\mathbb{C}))$. However, $\rho$ is a good representation.\end{example}

\begin{example}\label{orbnotalg}
Let $\Gamma=\mathbb{Z}/3\times \mathbb{Z}/3$ and $G=PSL(3,\mathbb{C})$. Define $\rho$ to be an isomorphism between $\Gamma$ and the unique (up to conjugation)  abelian irreducible group of $G$. Being locally rigid, $\rho$ is scheme smooth. Equation \ref{dimesptgt} implies that $\dim\left( T_{[\rho]}\mathfrak{X}(\Gamma,G)\right)=0$ and therefore $[\rho]$ is necessarily scheme smooth as well. However, $\rho$ is an abelian irreducible representation and, in particular, is bad.\end{example}

\begin{remark}  In the proof of Proposition \ref{singvarcar}, we used the hypothesis that $\Gamma$ is a free group or a closed surface group for two different things. The first was to insure that irreducible representations are scheme smooth (this is what fails in  Example \ref{algnotorb}). The second was to insure that    cohomology groups of $\Gamma$ into $\mathfrak{d}_k$ are big enough to use Lemma \ref{cyclicaction} (this is what trivially fails in Example \ref{orbnotalg} because $\Gamma$ is finite). The modular group in $PSL(2,\mathbb{C})$ leads to a non-locally rigid example where the second condition fails, see Proposition \ref{PSL(2,Z)}. \end{remark}

\begin{proposition}\label{PSL(2,Z)}

$\chi^i(PSL(2,\mathbb{Z}),PSL(p,\mathbb{C}))$ is a manifold for any  prime $p>3$.

 The singular locus of $\chi^i(PSL(2,\mathbb{Z}),PSL(3,\mathbb{C}))$ is a singleton.  Its unique point  is not scheme smooth.

The singular locus of $\chi^i(PSL(2,\mathbb{Z}),PSL(2,\mathbb{C}))$ is a singleton. Its unique point   is scheme smooth. 

\end{proposition}

\begin{proof}
 First,  $PSL(2,\mathbb{Z})$ is isomorphic to $\mathbb{Z}/2*\mathbb{Z}/3=\langle a, b\mid a^2=1=b^3\rangle$. Therefore 
$$\Homb(PSL(2,\mathbb{Z}),PSL_p)=\Homb(\mathbb{Z}/2,PSL_p)\times \Homb(\mathbb{Z}/3,PSL_p)\text{.}$$
Since $\Homb(\mathbb{Z}/2,PSL_p)$ and $ \Homb(\mathbb{Z}/3,PSL_p)$ are both smooth (because any closed point is locally rigid whence scheme smooth), $\Homb(PSL(2,\mathbb{Z}),PSL_p)$ is also smooth and therefore the open subvariety $\Hom^i(PSL(2,\mathbb{Z}),PSL(p,\mathbb{C}))$ is also  smooth. As a result,   $\chi^i(PSL(2,\mathbb{Z}),PSL(p,\mathbb{C}))$ is an orbifold and conjugacy classes of good representations are necessarily scheme smooth. 

\bigskip

When   $p>3$, there is no normal subgroup of index $p$ in $PSL(2,\mathbb{Z})$. Equation \ref{decompbadlocus2}  directly implies that the singular locus is empty, whence the orbifold $\chi^i(PSL(2,\mathbb{Z}),PSL(p,\mathbb{C}))$ is actually a manifold.

 $p=3$. There is only one normal subgroup of index $3$ in $PSL(2,\mathbb{Z})$ (it is normally generated by $a$). Any bad morphism from $PSL(2,\mathbb{Z})$ in $PSL(3,\mathbb{C})$ will then   be conjugate to $\beta_3$ : $a\mapsto \overline{\begin{pmatrix}-1&&\\&-1&\\&&1\end{pmatrix}}$, $b\mapsto \overline{M_c}$ because $a$ needs to be of order $2$. By definition, there is a single point in $\chi^i_{Sing}(PSL(2,\mathbb{Z}),PSL(3,\mathbb{C}))$.

$p=2$. Likewise, there is  only one  conjugacy class of  bad representations from $PSL(2,\mathbb{Z})$ in $PSL(2,\mathbb{C})$ which is given by  $\beta_2$ : $a\mapsto \overline{M_c}$, $b\mapsto \overline{\begin{pmatrix}e^{\frac{2\sqrt{-1}\pi}{3}}&\\&e^{-\frac{2\sqrt{-1}\pi}{3}}\end{pmatrix}}$. There is a single point in $\chi^i_{Sing}(PSL(2,\mathbb{Z}),PSL(2,\mathbb{C}))$.

\bigskip

To study, the singularity  $[\beta_p]$ in the character variety one does the exact same thing as in Proposition \ref{singvarcar}. Since $\dim (H^1(PSL(2,\mathbb{Z}),\mathfrak{d}_{k,Ad\circ\beta_p}))=1$ for $p=2,3$ and $1\leq k\leq p-1$, we may apply  the criterion of Lemma \ref{cyclicaction} and Sikora's formula (Equation \ref{dimesptgt}). This leads to the wanted result. \end{proof}

 There are more general questions related to  the algebraic singularities of the (schematic) character variety.
\begin{itemize}

\item  Completely reducible representations which are not irreducible are usually algebraic singularities of free groups  character varieties (see for instance Theorem 3.21  in \cite{F-L}). But there are, in some cases, non irreducible representations whose conjugacy class is "accidentally" smooth on the (schematic) character variety. It is highlighted in Remark 3.22 in loc. cit., one could think of  $\chi(\mathbb{Z},SL(n,\mathbb{C}))=\mathbb{C}^{n-1}$ or $\chi(\mathbb{F}_2,SL(2,\mathbb{C}))=\mathbb{C}^3$ whose all points, even the non-irreducible ones, are smooth.

\item The method in the proof of Proposition \ref{singvarcar} virtually generalizes to any Fuchsian groups (by \textit{Fuchsian group}, we mean discrete subgroup $\Gamma$ of $PSL(2,\mathbb{R})$ such that $\mathbb{H}^2/\Gamma$ has finite volume, see \cite{Wei}). Indeed, for such groups, irreducible representations are scheme smooth in $PSL(p,\mathbb{C})$, which allows us to use Sikora's formula (Equation \ref{dimesptgt}). It would certainly be interesting to look more closely to non-Fuchsian examples (e.g. Example \ref{algnotorb}).  

\item Does Proposition \ref{singvarcar} remain true if  $PSL(p,\mathbb{C})$ is replaced by any complex reductive group $G$ (in the free group case, it is related to Conjecture 3.34 in \cite{F-L}) ?
 \end{itemize}

\begin{appendix}
\section{Group cohomology}\label{grocoh}

The aim of this section is to give a brief overview of the results needed in this paper for the cohomology of groups. For a complete  overview of the theorems and results about group cohomology,  the reader is advised to read \cite{Bro}. For this section, the modules will be noted additively (note that in Section \ref{diagbyfin}, the modules are multiplicative and additive in Section \ref{orbalgsing}). 

\bigskip

Let $G$ be a group and $M$ be an additive $G$-module. We begin with a few definitions for  cocycles of degree $1$ or $2$ (compare with  Chapter III, Paragraph 1, Example 3 in loc. cit.). 

The set of fixed points for the $G$-action of $M$ will be denoted $M^G$. 

 A \textit{$1$-cocycle} from $G$ to $M$ is a map $z:G\to M$ such that for $g,h\in G$, we have $z(gh)=z(g)+g\cdot z(h)$. 
 
 A \textit{$1$-coboundary} from $G$ to $M$ is a map $b:G\to M$ such that there is $m\in M$ verifying for $g\in G$,  $b(g)=m-g\cdot m$. 
 
 A \textit{$2$-cocycle} from $G$ to $M$ is a map $z:G^2\to M$ such that for $g,h,k\in G$, we have $z(g,h)=g\cdot z(h,k)- z(gh,k)+z(g,hk)$.
 
 A \textit{$2$-coboundary} from $G$ to $M$ is a map $b:G^2\to M$ such that there is a map $m:G\to M$ verifying for $g,h\in G$, $b(g,h)=m(g)+g\cdot m(h)-m(gh)$. 

\bigskip

For $i=1,2$. The set of $i$-cocycles (resp. $i$-coboundaries) is denoted $Z^i(G,M)$ (resp. $B^i(G,M)$). It is  straightforward to check that $B^i(G,M)$ is contained in $Z^i(G,M)$. Remark that they are both abelian subgroups of $C^i(G,M)$  (the set of maps from $G^i$ to $M$, also known as the set of $i$-cochains in this context). The \textit{$i$-th cohomology group} is the quotient  $H^i(G,M):=Z^i(G,M)/B^i(G,M)$.  If $z$ is a $i$-cocycle, $[z]$ denotes its image in    $H^i(G,M)$.

Remark that $H^i(G,M)$ is actually defined for all $i\in \mathbb{N}$ but we will only use the first and the second cohomology groups.

\bigskip

These low-degree cohomology groups are useful in algebra (read Chapter IV in loc. cit. for applications to group extensions), especially using cochains. In Section \ref{diagbyfin} and Section \ref{orbalgsing}, this is the link between the first cohomology group and semidirect product which is used  (see Remark \ref{similar}). The first example of computation, which directly follows from the definition, is an easy exercise which is left to the reader. 
\begin{lemma}\label{H1triv}

Let $G$ be a group and $M$ be a trivial $G$-module (that is $G$ acts trivially on $M$) then $H^1(G,M)=\Hom(G,M)$. 

\end{lemma}
In a list of four papers, Fox introduced free differential calculus (the second paper of this list is \cite{Fox}). It can be used to  compute low-degree cohomology of groups (see \cite{Gol} for instance) especially when the module is  a vector space. In general, computing  cohomology groups proves to be  difficult using its definition with cochains and one need to use a more abstract approach.

\bigskip

Let $G$ be  a cyclic group with generator $g$ and order $n$. We define two maps related to the $G$-module $M$ :
$$\Norme_M:\left|\begin{array}{lcc}M&\longrightarrow& M^G\\x&\longmapsto & \sum_{i=0}^{n-1}g^i\cdot x\end{array}\right. \text{ and } \Trace_M:\left|\begin{array}{lcc}M&\longrightarrow& M\\x&\longmapsto &x-g\cdot x\end{array}\right.$$
Then, we have in Chapter III, Paragraph 1, Example 2 in \cite{Bro} :
\begin{lemma}\label{H1cycl}

Let $G$ be a cyclic group with generator $g$ and order $n$ and $M$ be a $G$-module, then $H^1(G,M)$ is isomorphic to $\Ker(\Norme_M)/\im(\Trace_M)$. 

\end{lemma}
Although we will not use more than what is contained in this lemma, it should be acknowledged that  the cohomology  of cyclic groups is $2$-periodic and the lemma mentioned above is a simple consequence of this fact (see Chapter VI, Paragraph 9 in loc. cit. for more details).

\bigskip

The following lemma is particularly interesting in Section  \ref{orbalgsing} where the modules considered are $\mathbb{C}$-vector spaces (see Chapter III, Corollary 10.2 in loc. cit.).

\begin{lemma}\label{cohfin}

Let $G$ be a finite group and $M$ be a $G$-module. If the multiplication by $|G|$ is invertible in $M$ (e.g. if $M$ is a vector space over a field of characteristic $0$) then $H^1(G,M)=H^2(G,M)=0$.

\end{lemma}

The last result we will need is a way to relate the cohomology of a group with the cohomology of its normal subgroups. Before stating the proposition, we need to make a few comments. Start with $G$ a group and a normal subgroup  $N$ of $G$. Let $Q$ be the quotient group $G/N$ with its natural projection $p:G\rightarrow Q$. If $M$ is a $G$-module then  $M$ is also a $N$-module (by simply restricting the action) and  $M^N$ is a $Q$-module (since $M^N$ is both a $G$-module and  a trivial $N$-module).

For $q\in Q$, we arbitrarily choose $x_q\in G$ such that $x_q$ mod $N= q$.  If $z$ belongs to $Z^1(N,M)$, then one can see that the map $z_q: n\mapsto x_q\cdot z(x_q^{-1}nx_q)$ is also an element of $Z^1(N,M)$. Furthermore, if $z$ is in $B^1(N,M)$, then $z_q$ is also in  $B^1(N,M)$. Therefore, we have a map from $H^1(N,M)$ to itself sending $[z]$ to $[z_q]$. One can check that this defines an action of $Q$ on $H^1(N,M)$ which does not depend on the lifts $x_q$ chosen.

\bigskip

If $z$ is a $1$-cocycle from $G$ to $M$, then its restriction $z'$ to $N$ is also a $1$-cocycle. Furthermore, $[z']$ can be shown to be a fixed points for the $Q$-action on $H^1(N,M)$. We denote $\Res[z]:=[z']$. This defines a map  $\Res: H^1(G,M)\to H^1(N,M)^Q$ called the \textit{restriction map}.

If $w$ is  a $1$-cocycle from $Q$ to $M^N$, then  its push-forward by $p$,  $z:=w\circ p: G\to M$   belongs to $Z^1(G,M)$ and we denote $\Inf[w]:=[z]$. This defines a map $\Inf: H^1(Q,M^N)\to H^1(G,M)$ called the \textit{inflation map}.

\bigskip

The result we are interested in is the decomposition of the cohomology of $G$ in $M$  using the cohomology of $Q$ and $N$. The decomposition is called the Hochschild-Serre Spectral Sequence see Chapter VII, Theorem 6.3 in \cite{Bro} (remark that the spectral sequence is given for the homology but also works for the cohomology). For us, the most useful part   of this spectral sequence are its first terms which give the Inflation-Restriction sequence, see Chapter VII, Corollary 6.4 in loc. cit. :
\begin{proposition}\label{InfRes}
Let $G$ be a group and $M$ be a $G$-module.  Assume we have the following exact sequence $1\rightarrow N\rightarrow G\rightarrow Q\rightarrow 1$. Then there is a map $T: H^1(N,M)^Q\to H^2(Q,M^N)$ called the \textit{Transgression map} such that the following exact sequence is exact :

$$\xymatrix{1\ar[r]&H^1(Q,M^N)\ar[r]^{\Inf}&H^1(G,M)\ar[r]^{\Res}&H^1(N,M)^Q\ar[r]^{T}&H^2(Q,M^N)}$$
\end{proposition}
\begin{remark}\label{DHWexplicit}
It is possible to explicitly write down the Transgression map $T$ defined in the proposition, see Paragraph  10.2 in \cite{D-H-W}.  \end{remark}

\end{appendix}

\end{document}